\documentclass[10pt]{amsart}

\usepackage[all]{xy}
\usepackage{xspace}
\usepackage[foot]{amsaddr}
\usepackage{graphicx}
\usepackage{longtable}
\usepackage{mathtools}
\usepackage{latexsym}
\usepackage{stmaryrd}
\usepackage{hyperref}
\usepackage{fullpage}
\usepackage{mathrsfs}
\usepackage{amsfonts,amsmath,amscd,amssymb}
\input{xy}
\xyoption{all}
\xyoption{curve}

\usepackage{array}

\hypersetup{
	colorlinks=true,
	linkcolor=blue,
	filecolor=magenta,      
	urlcolor=cyan,
}

\newcommand{\bC}{{\mathbb C}}

\newcommand{\bK}{{\mathbb K}}
\newcommand{\bF}{{\mathbb F}}

\newcommand{\bQ}{{\mathbb Q}}
\newcommand{\bZ}{{\mathbb Z}}
\newcommand{\bN}{{\mathbb N}}

\newcommand{\fB}{{\mathfrak B}}
\newcommand{\fc}{{\mathfrak c}}
\newcommand{\fg}{{\mathfrak g}}
\newcommand{\fl}{{\mathfrak l}}
\newcommand{\fh}{{\mathfrak h}}
\newcommand{\fm}{{\mathfrak m}}
\newcommand{\fn}{{\mathfrak n}}
\newcommand{\fo}{{\mathfrak o}}
\newcommand{\fp}{{\mathfrak p}}
\newcommand{\fr}{{\mathfrak r}}
\newcommand{\fs}{{\mathfrak s}}

\newcommand{\cN}{{\mathcal N}}

\newcommand{\cL}{{\mathcal L}}

\newcommand{\sG}{{\mathscr G}}
\newcommand{\fb}{{\mathfrak b}}

\newcommand{\codim}{{{\mbox{\rm codim}}}}

\newcommand{\Ext}{{{\mbox{\rm Ext}}}}

\newcommand{\Irr}{{{\mbox{\rm Irr}}}}
\newcommand{\Id}{{{\mbox{\rm Id}}}}

\newcommand{\cen}{{{\mbox{\rm cen}}}}

\newcommand{\twopartdef}[4]{\left\{
	\begin{array}{ll}
		#1 & \mbox{if } #2 \\
		#3 & \mbox{if } #4
	\end{array}
	\right.}

\newtheorem{theorem}{Theorem}[section]
\newtheorem{prop}[theorem]{Proposition}

\newtheorem{cor}[theorem]{Corollary}

\newtheorem*{conj}{Conjecture}

\theoremstyle{remark}
\newtheorem{rmk}{Remark}

\begin{document}
\title[Humphreys' Conjecture]{A Note on Humphreys' Conjecture on Blocks}
%to an Algebraic Group}
\author{Matthew Westaway}
\email{M.P.Westaway@bham.ac.uk}
\address{School of Mathematics, University of Birmingham, Birmingham, B15 2TT, UK}
%\thanks{The author was supported during this research by EPSRC grant EP/R018952/1 and later by a research fellowship from the Royal Commission for the Exhibition of 1851.}
\date{\today}
\subjclass[2020]{Primary: 17B35, 17B50, Secondary: 16D70, 17B10}
\keywords{Reduced enveloping algebra, Block decomposition, Lie algebra, Humphreys' conjecture}
  
\begin{abstract}
	Humphreys' conjecture on blocks parametrises the blocks of reduced enveloping algebras $U_\chi(\fg)$, where $\fg$ is the Lie algebra of a reductive algebraic group over an algebraically closed field of characteristic $p>0$ and $\chi\in\fg^{*}$. It is well-known to hold under Jantzen's standard assumptions. We note here that it holds under slightly weaker assumptions, by utilising the full generality of certain results in the literature. We also provide a new approach to prove the result for $\fg$ of type $G_2$ in characteristic 3, a case in which the previously mentioned weaker assumptions do not hold. This approach requires some dimensional calculations for certain centralisers, which we conduct in the Appendix for all the exceptional Lie algebras in bad characteristic.
%	
%	
%We discuss the modular representation theory of reductive Lie algebras $\fg$ over fields of bad characteristic. Specifically, we verify Humphreys' conjecture on blocks when $\fg=\Lie(G)$ is the Lie algebra of a reductive algebraic group with simply-connected derived subgroup and with the property that there exists a $G$-equivariant isomorphism $\fg\xrightarrow{\sim}\fg^{\ast}$. We also prove it when $G$ is the almost-simple simply-connected algebraic group of type $G_2$ in characteristic 3. Finally, we prove some dimensional results for certain modules over Lie algebras of exceptional type in bad characteristic. 
\end{abstract}

\maketitle

\section{Introduction}\label{sec1}

One of the most powerful tools in representation theory is the notion of the {\bf block decomposition}. Given a finite-dimensional $\bK$-algebra $A$, the block decomposition of $A$ gives a partition of the set of irreducible $A$-modules. We may then study the representation theory of $A$ block-by-block. In particular, if we well-understand one block (say, a block containing a trivial module) then we can often use {\bf translation functors} to gain insight into the structure of other blocks.

If $G$ is an algebraic group over an algebraically closed field $\bK$ of characteristic $p>0$ and $\fg$ is its Lie algebra, then we may form, for each $\chi\in\fg^{*}$, the {\bf reduced enveloping algebra} $U_\chi(\fg)$. This is a finite-dimensional $\bK$-algebra which is important to the representation theory of $\fg$, and so we would like to understand its blocks. The leading result in this direction is {\bf Humphreys' conjecture on blocks}.

\begin{conj}[Humphreys' conjecture on blocks]
	Suppose $G$ is reductive and let $\chi\in\fg^{*}$ be nilpotent. Then there exists a natural bijection between the blocks of $U_\chi(\fg)$ and the set $\Lambda_\chi/W_{\bullet}$. In particular, $$\left\vert\{\mbox{Blocks of}\,\,\, U_\chi(\fg)\}\right\vert=\left\vert \Lambda_\chi/W_{\bullet}\right\vert.$$
\end{conj}

Here, $\Lambda_\chi$ is a certain finite subset of $\fh^{*}$, where $\fh$ is the Lie algebra of a maximal torus $T$ of $G$, and $W$ is the Weyl group of $(G,T)$, which acts on $\fh^{*}$ via the dot-action and thus induces an equivalence relation on $\Lambda_\chi$. The requirement that $\chi$ is nilpotent means that $\chi$ vanishes on the Lie algebra $\fb$ of a Borel subgroup $B$ of $G$ (which we may assume contains $T$).

This conjecture was proved by Humphreys \cite{Hu3.1} in 1971 for $\chi=0$, subject to the requirements that $G$ be semisimple and that $p>h$, where $h$ is the Coxeter number of $(G,T)$. Humphreys then extended the result further to $\chi$ in so-called standard Levi form in 1998 in \cite{Hu.1} (the paper \cite{Hu.1} doesn't explicitly state what assumptions are being made, but the argument holds for any connected reductive algebraic group whose derived group is simply-connected). Under three assumptions (which we will call Jantzen's standard assumptions \cite{J1.1,J2.1} and denote (A), (B) and (C)), the conjecture was then proved by Brown and Gordon in \cite{BG.1} for all $\chi\in\fg^{*}$ when $p>2$, and then improved by Gordon in \cite{Go.1} to include the $p=2$ case (so long as (A), (B) and (C) still hold). In fact, under assumptions (A), (B) and (C), Humphreys' conjecture on blocks allows us to count the number of blocks of $U_\chi(\fg)$ for {\em all} $\chi\in\fg^{*}$, as these assumptions are sufficient to reduce the computation to the case of nilpotent $\chi$ (see \cite{FP1.1}, also Remark~\ref{NilpRed}, {\em infra}). Furthermore, Braun \cite{Br.1} recently proved the conjecture for $\fg=\fs\fl_n$ with $p\vert n$, where assumptions (A) and (B) hold but (C) doesn't. In this case, however, the restriction to nilpotent $\chi$ is necessary, as the analogous result for semisimple $\chi$ was shown in \cite{Br.1} to fail when $p=n=3$.

Let us now explain Jantzen's standard assumptions. These are: (A) that the derived group of $G$ is simply-connected; (B) that the prime $p$ is good for $G$; and (C) that there exists a non-degenerate $G$-invariant bilinear form on $\fg$. The primes that are not good for a given $G$ can be listed explicitly (and are all less that or equal to $5$), and the existence of a non-degenerate $G$-invariant bilinear form on $\fg$ holds whenever $\fg$ is simple.

The question motivating this note is: what happens to Humphreys' conjecture on blocks for nilpotent $p$-characters if we remove assumptions (B) and/or (C)? We see in Section~\ref{sec3} that there is a natural surjection $f:\{\mbox{Blocks of}\,\,\, U_\chi(\fg)\}\to\Lambda_\chi/W_{\bullet}$ under only assumption (A). It turns out that this can be deduced from the literature \cite{J2.1,KW.1}. Furthermore, we show in Theorem~\ref{BlockNumb} that the known proof of the injectivity of $f$ works without assumption (B). This therefore confirms Humphreys' conjecture for the almost-simple groups over algebraically closed fields of the following bad characteristics:

\begin{cor}\label{ASGps}
	Let $G$ be an almost-simple group over an algebraically closed field $\bK$ of bad characteristic $p>0$. Then Humphreys' conjecture holds for $G$ when $p=2$ and $G$ is of type $E_6, E_8$ or $G_2$, when $p=3$ and $G$ is of type $E_7, E_8$ or $F_4$, and when $p=5$ and $G$ is of type $E_8$.
\end{cor}

We also provide a different approach to the proof of the injectivity in Proposition~\ref{prop1}, which demonstrates that injectivity in fact holds whenever there exists a collection of irreducible modules of a certain nice form (namely, which are so-called {\bf baby Verma modules}). Premet's theorem \cite{Pr1.1} shows the existence of such irreducible modules under assumptions (A), (B) and (C), and we observe in Corollary~\ref{BlockG23} that the existence also holds for the almost-simple algebraic group of type $G_2$ in characteristic 3 (where assumption (C) fails). This thus proves Humphreys' conjecture on blocks for $G_2$ in characteristic 3, which could not be deduced using the previous approach.

In the Appendix, we conduct some calculations with a view to finding other examples where these irreducible modules exist. Unfortunately, the calculations do not lead to further examples, but we hope the calculations are interesting in their own right, as they demonstrate divisibility bounds for irreducible modules for certain nice $\chi$ and small primes.

{\bf Statements and Declarations:} The author was supported during this research by the Engineering and Physical Sciences Research Council, grant EP/R018952/1, and later by a research fellowship from the Royal Commission for the Exhibition of 1851.

{\bf Acknowledgments:} The author would like to thank Ami Braun and Dmitriy Rumynin for suggesting this question, and Simon Goodwin for engaging in many useful discussions regarding this subject and comments on earlier versions of this paper.

\section{Preliminaries on Lie algebras}\label{sec2}

Throughout this note we work with a connected algebraic group $G$ over an algebraically closed field $\bK$ of characteristic $p>0$. More precise assumptions on $G$ are given section-by-section, but it is always at least a reductive algebraic group with simply-connected derived subgroup. Inside $G$, we fix a maximal torus $T$ and a Borel subgroup $B$ of $G$ containing $T$. Write $X(T)$ for the character group of $T$, $Y(T)$ for the cocharacter group of $T$, and $\langle\cdot,\cdot\rangle:X(T)\times Y(T)\to\bZ$ for the natural pairing. We write $\fg$ for the Lie algebra of $G$, $\fb$ for the Lie algebra of $B$ and $\fh$ for the Lie algebra of $T$. As Lie algebras of algebraic groups these are all restricted, so come equipped with $p$-th power maps $\fg\to\fg$ (resp. $\fb\to\fb$, $\fh\to\fh$) written $x\mapsto x^{[p]}$ .

Set $\Phi$ to be the root system of $G$ with respect to $T$, $\Phi^{+}$ to be the positive roots corresponding to $B$ and $\Pi$ to be the simple roots. For $\alpha\in \Phi$ we set $\alpha^\vee\in Y(T)$ to be the corresponding coroot, and we write $\fg_\alpha$ for the root space of $\alpha$ in $\fg$. We then define $\fn^{+}=\bigoplus_{\alpha\in\Phi^{+}}\fg_\alpha$ and $\fn^{-}=\bigoplus_{\alpha\in\Phi^{+}}\fg_{-\alpha}$, so $\fg=\fn^{-}\oplus\fh\oplus\fn^{+}$. For $\alpha\in \Phi$ we define $h_\alpha\coloneqq d\alpha^\vee(1)\in\fh$, and we choose $e_\alpha\in \fg_\alpha$ and $e_{-\alpha}\in\fg_{-\alpha}$ so that $[e_\alpha,e_{-\alpha}]=h_\alpha$ (see, for example, \cite{J4.1} for more details on this procedure). We also choose a basis $h_1,\ldots,h_d$ of $\fh$ with the property that $h_i^{[p]}=h_i$ for all $1\leq i\leq d$.

Set $W$ to be the Weyl group of $\Phi$, which acts naturally on $X(T)$ and $\fh^{*}$. We fix $\rho\in X(T)\otimes_{\bZ}\bQ$ to be the half-sum of positive roots in $\Phi$. This then allows us to define the dot-action of $W$ on $X(T)$ as $w\cdot\lambda=w(\lambda+\rho)-\rho$ (noting that this action makes sense even if $\rho\notin X(T)$). When $\rho\in X(T)$, $d\rho(h_\alpha)=1$ for all $\alpha\in \Pi$. If $\rho\notin X(T)$, we may still define $d\rho\in\fh^{*}$ such that $d\rho(h_\alpha)=1$ for all $\alpha\in\Pi$, since the derived subgroup being simply-connected implies that these $h_\alpha$ are linearly independent in $\fh$. We may therefore define the dot action on $\fh^{*}$ similarly to how it was defined on $X(T)$. When we wish to specify that $W$ is acting through the dot-action, we may write $W_\bullet$ instead of $W$.

We write $U(\fg)$ for the universal enveloping algebra of $\fg$. We write $Z_p$ for the central subalgebra of $U(\fg)$ generated by all $x^p-x^{[p]}$ with $x\in\fg$, which we call the {\bf $p$-centre} of $U(\fg)$. Given $\chi\in\fg^{*}$, we write $U_\chi(\fg)$ for the reduced enveloping algebra $U_\chi(\fg)\coloneqq U(\fg)/\langle x^p-x^{[p]}-\chi(x)^p \,\vert\, x\in\fg\rangle$. Each irreducible $\fg$-module is finite-dimensional \cite[Theorem A.4]{J1.1} and so, by Schur's lemma, each irreducible $\fg$-module is a $U_\chi(\fg)$-module for some $\chi\in\fg^{*}$. For $\chi\in\fg^{*}$, we recall that the centraliser of $\chi$ in $\fg$ is defined as $c_{\fg}(\chi)\coloneqq \{x\in\fg\,\vert\,\chi([x,\fg])=0\}.$

The adjoint action of $G$ on $\fg$ induces the coadjoint action of $G$ on $\fg^{*}$, and if $\chi,\mu\in\fg^{*}$ lie in the same coadjoint $G$-orbit then $U_\chi(\fg)\cong U_\mu(\fg)$. The derived group of $G$ being simply-connected implies (see \cite{J2.1,KW.1}) that any $\mu\in\fg^{*}$ lies in the same $G$-orbit as some $\chi\in\fg^{*}$ with $\chi(\fn^{+})=0$. Putting these two observations together, we always assume $\chi(\fn^{+})=0$ throughout this paper. 

We can define, for each $\lambda\in\fh^{*}$, a one-dimensional $\fb$-module $\bK_\lambda$ on which $\fn^{+}$ acts as zero and $\fh$ acts via $\lambda$. The assumption that $\chi(\fn^{+})=0$ means that $\bK_\lambda$ is a $U_\chi(\fb)$-module if and only if $\lambda\in\Lambda_\chi$, where \begin{equation*}
	\begin{split}
		\Lambda_\chi & \coloneqq \{\lambda\in\fh^{*}\mid\lambda(h)^p-\lambda(h^{[p]})=\chi(h)^p\,\,\mbox{for all}\,\, h\in\fh\} \\ & = \{\lambda\in\fh^{*}\mid\lambda(h_i)^p-\lambda(h_i)=\chi(h_i)^p\,\,\mbox{for all}\,\, 1\leq i\leq d\}
	\end{split}
\end{equation*} and that all irreducible $U_\chi(\fb)$-modules are of this form. We therefore may define the {\bf baby Verma module} $Z_\chi(\lambda)=U_\chi(\fg)\otimes_{U_{\chi}(\fb)}\bK_\lambda$, a $U_\chi(\fg)$-module of dimension $p^N$, where $N=\left\vert\Phi^{+}\right\vert$. Every irreducible $U_\chi(\fg)$-module is the quotient of some baby Verma module (see \cite[Lem. B.4]{J1.1}).

Since $W_\bullet$ acts on $\fh^{*}$, we may define an equivalence relation on $\Lambda_\chi$ by setting $\lambda\sim\mu$ if and only if there exists $w\in W$ with $w\cdot\lambda=\mu$. We write $\Lambda_\chi/W_{\bullet}$ for the set of equivalence classes of $\Lambda_\chi$ under this relation.

If $\chi(\fb)=0$ then $\Lambda_\chi=\Lambda_0=\{d\lambda\in\fh^{*}\,\vert\,\lambda\in X(T)\}=X(T)/pX(T)$. In this case, $W_\bullet$ in fact acts on $\Lambda_\chi$, so $\Lambda_\chi/W_\bullet$ is the set of $W_\bullet$-orbits for this action. The condition that $\chi(\fb)=0$ is sufficiently important in this paper that we make the definition $$\fb^{\perp}\coloneqq\{\chi\in\fg^{*}\,\vert\,\chi(\fb)=0\}.$$

We say that $\chi\in\fb^{\perp}$ is in {\bf standard Levi form} if there exists a subset $I\subseteq\Pi$ of simple roots such that $$\chi(e_{-\alpha})=\twopartdef{1}{\alpha\in I,}{0}{\alpha\in\Phi^{+}\setminus I.}$$ If $I=\Pi$ we say that $\chi$ is {\bf regular nilpotent in standard Levi form}. In general, we say $\chi\in\fg^{*}$ is {\bf regular nilpotent} if it is in the same $G$-orbit as the $\mu\in\fg^{*}$ which is regular nilpotent in standard Levi form.

\section{Preliminaries on Blocks}\label{sec3}

Let us briefly recall the definition of the {\bf blocks} of a finite-dimensional $\bK$-algebra $A$ (one can find more details in \cite[I.16, III.9]{BGo.1}, for example). We say that one irreducible $A$-module $M$ is {\bf linked to} another irreducible $A$-module $N$ if $\Ext^1(M,N)\neq 0$. This is not an equivalence relation, but we may refine it to one. The equivalence classes under the resulting equivalence relation are then called the {\bf blocks} of $A$.

In this note, we are concerned with the case of $A=U_\chi(\fg)$ with $\chi\in\fb^\perp$. Under assumptions (A), (B) and (C) the results in this section are well-known -- for example, they are contained within the proof of Proposition C.5 in \cite{J1.1}. Nonetheless, we recall them to highlight when assumptions (A), (B) and (C) are or are not necessary. Remember from Section~\ref{sec2} that each irreducible $U_\chi(\fg)$-module is a quotient of a baby Verma module $Z_\chi(\lambda)$, and thus all irreducible $U_\chi(\fg)$-modules appear as composition factors of baby Verma modules. Recall also that the {\bf Grothendieck group} $\sG (U_\chi(\fg))$ of the category of finite-dimensional $U_\chi(\fg)$-modules is the abelian group generated by symbols $[M]$, for $M$ running over the collection of all finite-dimensional $U_\chi(\fg)$-modules, subject to the relation that $[P]+[N]=[M]$ if 
$$0\to P\to M\to N\to 0$$ is a short exact sequence of $U_\chi(\fg)$-modules. It is clear that in $\sG (U_\chi(\fg))$ we have, for $\lambda\in\Lambda_0$, $$[Z_\chi(\lambda)]=\sum_{\tiny L\in\Irr(U_\chi(\fg))} [Z_\chi(\lambda):L][L],$$ where $\Irr(U_\chi(\fg))$ is the set of isomorphism classes of irreducible $U_\chi(\fg)$-modules and $[Z_\chi(\lambda):L]$ indicates the composition multiplicity of $L$ in $Z_\chi(\lambda)$.

We wish to define the map
$$f:\{\mbox{Blocks of}\,\, U_\chi(\fg)\}\to \{[Z_\chi(\lambda)]\,\vert\,\lambda\in \Lambda_0\}\subseteq \sG(U_\chi(\fg)),$$ as follows. Let $\fB$ be a block of $U_\chi(\fg)$, and let $E$ be an irreducible module in this block. There must exist $\lambda\in \Lambda_0$ such that $E$ is a quotient of $Z_\chi(\lambda)$. We then define $f(\fB)=[Z_\chi(\lambda)]$. 

For this to be well-defined, it is necessary to see that it does not depend on our choice of $E\in\fB$ or on our choice of $Z_\chi(\lambda)\twoheadrightarrow E$. For this, we note that $U(\fg)^G\subseteq Z(U(\fg))$ acts on the baby Verma module $Z_\chi(\lambda)$ via scalar multiplication as follows. Under the assumption that the derived group of $G$ is simply-connected (assumption (A)), the argument of Kac and Weisfeiler in \cite[Th. 1]{KW.1} (c.f. \cite[Th. 9.3]{J2.1}) shows that there exists an isomorphism $\pi:U(\fg)^G\to S(\fh)^{W_{\bullet}}$, where the dot-action on $S(\fh)$ is obtained by identifying $S(\fh)$ with the algebra $P(\fh^{*})$ of polynomial functions on $\fh^{*}$ and then defining $(w\cdot F)(\lambda)=F(w^{-1}\cdot\lambda)$ for $w\in W$, $F\in P(\fh^{*})$ and $\lambda\in\fh^{*}$. This isomorphism allows us, as in \cite{J2.1}, to define a homomorphism $\cen_{\lambda}:U(\fg)^G\to\bK$ which sends $u\in U(\fg)^G$ to $\pi(u)(\lambda)$, viewing $\pi(u)$ as an element of $P(\fh^{*})$. Then $U(\fg)^G$ acts on $Z_\chi(\lambda)$ via the character $\cen_\lambda$, for $\lambda\in \Lambda_0$. 

If $E$ and $E'$ lie in the same block then it is easy to see that $U(\fg)^G$ must act the same on both modules, and if $Z_{\chi}(\lambda_E)\twoheadrightarrow E$ and $Z_{\chi}(\lambda_{E'})\twoheadrightarrow E'$ then $U(\fg)^G$ acts on $E$ via $\cen_{\lambda_E}$ and on $E'$ by $\cen_{\lambda_{E'}}$. Thus, $\cen_{\lambda_E}=\cen_{\lambda_{E'}}$ and so, as in \cite[Cor. 9.4]{J2.1} (see also \cite[Th. 2]{KW.1}), we have $\lambda_E\in W_\bullet \lambda_{E'}$. One may then observe, using \cite[C.2]{J1.1}, that $[Z_\chi(\lambda_E)]=[Z_{\chi}(\lambda_{E'})]$. This shows that $f$ is well-defined. Furthermore, $f$ is clearly surjective (just take the block containing an irreducible quotient of the desired $Z_{\chi}(\lambda)$).

The above discussion also shows that $[Z_\chi(\lambda)]\cong [Z_\chi(\mu)]$ if and only if $\lambda\in W_\bullet\mu$. Thus, there is a bijection $$\{[Z_\chi(\lambda)]\,\vert\,\lambda\in \Lambda_0\}\leftrightarrow \Lambda_0/W_\bullet.$$ In particular, we get the following proposition (which also may more-or-less be found in \cite[C.5]{J1.1}), observing that at no point thus far have we required assumptions (B) or (C). 

\begin{prop}\label{Lower}
	Let $G$ be a connected reductive algebraic group over an algebraically closed field $\bK$ of characteristic $p>0$, with simply-connected derived subgroup, and let $\chi\in\fb^\perp$. Then there exists a natural surjection between the set of blocks of $U_\chi(\fg)$ and the set $\Lambda_\chi/W_{\bullet}=\Lambda_0/W_{\bullet}$. In particular, $$\left\vert\{\mbox{Blocks of}\,\,\, U_\chi(\fg)\}\right\vert\geq\left\vert \Lambda_0/W_{\bullet}\right\vert.$$
\end{prop}

\begin{rmk}
	We have used in the above argument the fact that, when assumption (A) holds, there exists an isomorphism $U(\fg)^G\xrightarrow{\sim}S(\fh)^{W_\bullet}$. This result dates back to Kac and Weisfeiler \cite{KW.1}, who proved it for connected almost-simple algebraic groups under the assumption that $G\neq SO_{2n+1}(\bK)$ when $p=2$.\footnote{In \cite[Th. 1]{KW.1} it is required that either $p\neq 2$ or $\rho\in X(T)$, where $\rho$ is the half sum of positive roots. This is then generalised to the given assumptions in \cite[Th. 1 BIS]{KW.1}. The $W$-action used in the latter theorem can be easily seen to be the same as the dot-action we are using.} According to Janzten \cite[Rem. 9.3]{J2.1}, the argument of Kac and Weisfeiler holds for reductive $\fg$ whenever assumption (A) holds. Jantzen further gives an argument \cite[9.6]{J2.1} using reduction mod $p$ techniques which holds under his standard assumptions. In fact, slightly weaker assumptions are sufficient: assumption (B) is only needed to ensure $p$ is not a so-called torsion prime of $\Phi^\vee$ (in the sense of \cite[Prop. 8]{Dem.1}), which is also satisfied for the bad prime 3 in case $G_2$, while assumption (C) is only needed to ensure that the (derivatives of the) simple roots are linearly independent in $\fh^{*}$, which is also satisfied for $p=2$ in type $F_4$ and $p=3$ in type $G_2$. In particular, the argument of Kac-Weisfeiler is unnecessary for our later result (Corollary~\ref{BlockG23}) that Humphreys' conjecture on blocks holds for the almost-simple algebraic group of type $G_2$ in characteristic 3.
\end{rmk}

\section{Upper bound}

Humphreys' conjecture on blocks claims that the map $f$ defined in the previous section is, in fact, a bijection. What remains, therefore, is to show that  $$\left\vert\{\mbox{Blocks of}\,\,\, U_\chi(\fg)\}\right\vert\leq\left\vert \Lambda_0/W_{\bullet}\right\vert.$$ Gordon \cite{Go.1} has shown that this inequality holds under assumptions (A), (B) and (C), and a similar argument is reproduced in \cite[C.5]{J1.1}. We give a version of this argument here in order to observe that it does not require assumption (B), and to highlight where assumption (C) is necessary:

The discussion in Section~\ref{sec3} shows that $U_\chi(\fg)$ has $\left\vert \Lambda_0/W_{\bullet}\right\vert$ blocks if, for each $\lambda\in\Lambda_0$, all composition factors of the baby Verma module $Z_\chi(\lambda)$ lie in the same block. This property holds for the $\mu\in\fg^{*}$ which is regular nilpotent in standard Levi form, since the corresponding baby Verma module has a unique maximal submodule and so is indecomposable, and it is well-known that all composition factors of an indecomposable module lie in the same block. Therefore $U_\chi(\fg)$ has $\left\vert \Lambda_0/W_{\bullet}\right\vert$ blocks for all $\chi$ in the $G$-orbit of $\mu$. 

Suppose now that the intersection of $\fb^\perp$ with the $G$-orbit of $\mu$ is dense in $\fb^\perp$. By \cite[Prop. 2.7]{Ga.1}, $$D_{\left\vert \Lambda_0/W_{\bullet}\right\vert}\coloneqq\{\chi\in\fb^{\perp}\,\vert\,U_\chi(\fg)\,\,\mbox{has at most}\,\,\left\vert \Lambda_0/W_{\bullet}\right\vert\,\,\mbox{blocks}\}$$ is closed in $\fb^{\perp}$. Since $(G\cdot\mu)\cap\fb^{\perp}\subseteq D_{\left\vert \Lambda_0/W_{\bullet}\right\vert}$, Humphreys' conjecture on blocks would follow.
	
When can we say that $(G\cdot\mu)\cap\fb^\perp$ is dense in $\fb^\perp$? Well, if there exists a $G$-equivariant isomorphism $\Theta:\fg\xrightarrow{\sim}\fg^{*}$, we can set $y\coloneqq\Theta^{-1}(\mu)$. Then \cite[6.3, 6.7]{J3.1} (which make no assumptions on $p$) establish that the $G$-orbit of $y$ is dense in the nilpotent cone $\cN$ of $\fg$, and so the $G$-orbit of $\mu$ is dense in $\Theta(\cN)$. Thus, $(G\cdot\mu)\cap \fb^\perp$ is dense in $\fb^{\perp}$, and so (cf. \cite[Th. 3.6]{Go.1}) under assumptions (A) and (C) we get Humphreys' conjecture of blocks:

\begin{theorem}\label{BlockNumb}
	Let $G$ be a connected reductive algebraic group over an algebraically closed field $\bK$ of characteristic $p>0$, with simply-connected derived subgroup. Suppose that there exists a $G$-module isomorphism $\Theta:\fg\xrightarrow{\sim}\fg^{*}$. Let $\chi\in\fb^{\perp}$. Then $$\left\vert\{\mbox{Blocks of}\,\,\, U_\chi(\fg)\}\right\vert=\left\vert \Lambda_\chi/W_{\bullet}\right\vert.$$
\end{theorem}

\begin{rmk}
	It is straightforward to see that this theorem implies Corollary~\ref{ASGps} from the Introduction.
\end{rmk}

\begin{rmk}\label{NilpRed}
	Suppose $\chi\in\fg^{*}$ with $\chi(\fn^{+})=0$. Under assumption (C), there exists a $G$-module isomorphism $\Theta:\fg\to\fg^{*}$, so we may fix $x\in\fg$ such that $\Theta(x)=\chi$. In $\fg$ it is well-known that $x$ has a (unique) Jordan decomposition $x=x_s+x_n$, where $x_s$ is semisimple, $x_n$ is nilpotent, and $[x_s,x_n]=0$, and thus we may define the Jordan decomposition $\chi=\chi_s+\chi_n$ where $\chi_s=\Theta(x_s)$ and $\chi_n=\Theta(x_n)$. In fact, Kac and Weisfeiler \cite[Th. 4]{KW.1} show that a Jordan decomposition of $\chi$ may be defined even when assumption (C) does not hold, so long as assumption (A) does instead: we say that $\chi=\chi_s+\chi_n$ is a Jordan decomposition if there exists $g\in G$ such that $g\cdot\chi_s(\fn^{+}\oplus\fn^{-})=0$, $g\cdot\chi_n(\fh\oplus\fn^{+})=0$, and, for $\alpha\in\Phi^{+}$, $g\cdot\chi(h_\alpha)\neq 0$ only if $g\cdot\chi(e_{\pm\alpha})=0$. Under assumptions (A) and (B), Friedlander and Parshall \cite{FP1.1} show that there is an equivalence of categories between $\{U_\chi(\fg)-\mbox{modules}\}$ and $\{U_\chi(\fc_{\fg}(\chi_s))-\mbox{modules}\}$ (the categories of finite-dimensional modules). It can then further be shown under those assumptions (see, for example, \cite[B.9]{J1.1}) that there is an equivalence of categories between $\{U_\chi(\fc_{\fg}(\chi_s))-\mbox{modules}\}$ and $\{U_{\chi_n}(\fc_{\fg}(\chi_s))-\mbox{modules}\}$. Under assumptions (A) and (B), this then often allows us to reduce representation-theoretic questions to the case of nilpotent $\chi$.
	
	When assumption (C) holds, we may do this for Humphreys' conjecture on blocks (we assume here that $\chi$ is chosen so that $g$ may be taken as $1$ in the definition of the Jordan decomposition, recalling that reduced enveloping algebras are unchanged by the coadjoint $G$-action on their corresponding $p$-character). The equivalence of categories between $\{U_\chi(\fg)-\mbox{modules}\}$ and $\{U_{\chi_n}(\fl)-\mbox{modules}\}$ (where $\fl\coloneqq \fc_{\fg}(\chi_s)$) clearly preserves the number of blocks of the respective algebras. Thus, Humphreys' conjecture on blocks for $(\fl,\chi_n)$ will imply it for $(\fg,\chi)$ if and only if $\left\vert\Lambda_\chi/W_{\bullet}\right\vert=\left\vert\Lambda_{\chi_n} /W'_\bullet\right\vert$, where $W'$ is the Weyl group corresponding to $\fl$. What is $W'$? Well, the root system for $\fl$ is $\{\alpha\in\Phi\mid \chi_s(h_\alpha)= 0\}$ so it is easy to see that $W'$ lies inside $W(\Lambda_\chi)$, the set of $w\in W$ which fix $\Lambda_\chi$ setwise (it is straightforward to see under our assumptions that it doesn't matter in defining this subgroup whether we consider the usual action or the dot-action of $W$, since $\rho\in \Lambda_0$). When assumption (C) holds, $W(\Lambda_\chi)$ is parabolic (see \cite[Lem. 7]{MR.1}, \cite[Prop. 1.15]{Hu4.1}), and so one can easily check that $W'=W(\Lambda_\chi)$ in this case (see \cite[Rem. 3.12(3)]{BG.1}). This then obviously implies that  $\left\vert\Lambda_\chi/W_{\bullet}\right\vert=\left\vert\Lambda_{\chi} /W'_\bullet\right\vert$, and so what remains is to show that $\left\vert\Lambda_{\chi} /W'_\bullet\right\vert=\left\vert\Lambda_{\chi_n} /W'_\bullet\right\vert$. One can check that there exists $\lambda\in \Lambda_\chi$ such that $w(\lambda)=\lambda$ for all $w\in W'=W(\Lambda_\chi)$. Then the map $\Lambda_\chi=\lambda+\Lambda_0\to \Lambda_0=\Lambda_{\chi_n}$, $\lambda+\tau\mapsto\tau$, induces a bijection $\Lambda_\chi/W'_\bullet\xrightarrow{\sim}\Lambda_{\chi_n}/W'_\bullet$ as required.
	
	Braun \cite[Th. 6.23, Ex. 6.25]{Br.1} has shown that when assumption (C) fails to hold, it can be the case that Humphreys' conjecture on blocks holds for nilpotent $\chi$ but fails for general $\chi$. Specifically, set $\fg=\fs\fl_3$, $p=3$, and choose $\chi\in\fs\fl_3^{*}$ such that $\chi(e_{11}-e_{22})=\chi(e_{22}-e_{33})\neq 0$ (using $e_{ij}$ for the usual basis elements of $\fg\fl_3$). Recalling that the Weyl group for $\fs\fl_3$ is the symmetric group $S_3$,  one can check that $W(\Lambda_\chi)=\{\Id,(1,2,3),(1,3,2)\}$ and so is not a parabolic subgroup of $W$. Thus, $W'\neq W(\Lambda_\chi)$ and so there can be linkages under $W$ which do not exist under $W'$.  In particular, choosing suitable $\chi$, one can use this to show that $\left\vert\Lambda_\chi/W_{\bullet}\right\vert<\left\vert\Lambda_{\chi_n} /W'_\bullet\right\vert$. Braun's argument then shows that the latter value is the number of blocks of $U_{\chi_n}(\fl)$ and so the number of blocks of $U_\chi(\fg)$. We note that this argument highlights that \cite[Lem. 7]{MR.1} requires the assumption that $p$ be very good for the root system.
\end{rmk}

The argument above highlights one approach to proving Humphreys' conjecture on blocks; namely, to obtain the desired result it suffices to find a dense subset of $\fb^{\perp}$ lying inside $D_{\left\vert \Lambda_\chi/W_{\bullet}\right\vert}$. Note that $\fb^\perp=\bK^N$, where $N=\left\vert\Phi^{+}\right\vert$, and recall that any non-empty open subset is dense in $\bK^N$ when it is equipped with the Zariski topology. For each $\lambda\in \Lambda_0$, define $$C_\lambda\coloneqq \{\chi\in\fb^{\perp}\,\vert\,\,\mbox{All composition factors of }\, Z_\chi(\lambda)\,\,\mbox{are in the same block of }\,U_\chi(\fg)\,\},$$ and define $$C\coloneqq\bigcap_{\lambda\in \Lambda_0}C_\lambda.$$ It is straightforward from the arguments in Section~\ref{sec3} to see that $C\subseteq D_{\left\vert \Lambda_\chi/W_{\bullet}\right\vert}$. Furthermore, if for each $\lambda\in \Lambda_0$ we can find a dense open subset $\widehat{C}_\lambda$ of $\fb^{\perp}$ with $\widehat{C}_\lambda\subseteq C_\lambda$, then $$\widehat{C}\coloneqq\bigcap_{\lambda\in \Lambda_0}\widehat{C}_\lambda$$ would be a dense open subset of $\fb^{\perp}$ contained in $C\subseteq D_{\left\vert \Lambda_\chi/W_{\bullet}\right\vert}$. Finding the desired $\widehat{C}_\lambda$ therefore provides an approach to proving Humphreys' conjecture on blocks, and in the rest of this section we explore one particular way of obtaining such $\widehat{C}_\lambda$.

For each $\lambda\in \Lambda_0$, consider the set $$S_\lambda\coloneqq\{\chi\in\fb^{\perp}\,\vert\,Z_\chi(\lambda)\,\,\mbox{is an  irreducible } U_\chi(\fg)\mbox{-module}\}.$$ It is remarked in \cite[C.6]{J1.1} that $S_\lambda$ is open in $\fb^\perp$. Specifically, if we define, for $s=1,\ldots,p^N-1$, the set $$N_{\lambda,s}=\{\chi\in\fb^{\perp}\,\vert\,Z_\chi(\lambda)\,\,\mbox{has a } U_\chi(\fg)\mbox{-submodule of dimension } s\},$$ then clearly $S_\lambda=\bigcap_{s=1}^{p^N-1}N_{\lambda,s}^c$ (where, for $X\subseteq \fb^\perp$, $X^c$ denotes $\fb^\perp\setminus X$). The openness of $S_\lambda$ then follows from the closure of each $N_{\lambda,s}$ in $\fb^\perp$ (which is proved in \cite[C.6]{J1.1}, and one can check that the proof doesn't use assumptions (B) or (C)).

\begin{prop}\label{prop1}
	Let $G$ be a connected reductive algebraic group over an algebraically closed field $\bK$ of characteristic $p>0$, with simply-connected derived subgroup. Let $\chi\in\fb^{\perp}$, and suppose that for each $\lambda\in \Lambda_0$ there exists $\mu_\lambda\in \fb^{\perp}$ such that $Z_{\mu_\lambda}(\lambda)$ is an irreducible $U_{\mu_\lambda}(\fg)$-module. Then $\left\vert\{\mbox{Blocks of}\,\, U_\chi(\fg)\}\right\vert=\left\vert \Lambda_\chi/W_{\bullet}\right\vert$.
\end{prop}
\begin{proof}
	Our assumption guarantees that each $S_\lambda$, for $\lambda\in\Lambda_0$, is non-empty. Each $S_\lambda$ is thus a dense open subset of $\fb^{\perp}$ and it is clear that $S_\lambda\subseteq C_\lambda$ for each $\lambda\in\Lambda_0$. We therefore have that $\bigcap_{\lambda\in \Lambda_0}S_\lambda$ is a dense open subset of $\fb^{\perp}$. Since $\bigcap_{\lambda\in \Lambda_0}S_\lambda\subseteq C\subseteq D_{\left\vert \Lambda_\chi/W_{\bullet}\right\vert}$ and $D_{\left\vert \Lambda_\chi/W_{\bullet}\right\vert}$ is closed in $\fb^\perp$, we conclude $\fb^\perp=D_{\left\vert \Lambda_\chi/W_{\bullet}\right\vert}$. Hence, $\left\vert\{\mbox{Blocks of}\,\, U_\chi(\fg)\}\right\vert\leq \left\vert \Lambda_\chi/W_{\bullet}\right\vert$ and, together with Proposition~\ref{Lower}, this gives the desired result.
\end{proof}

\begin{cor}\label{BlockG23}
	Suppose $G$ is the almost-simple simply-connected algebraic group of type $G_2$ over an algebraically closed field $\bK$ of characteristic $3$. If $\chi\in\fg^{*}$ satisfies $\chi(\fb)=0$, then $U_\chi(\fg)$ has exactly $\left\vert\Lambda_\chi/W_\bullet\right\vert=3$ blocks.
\end{cor}

\begin{proof}
	The calculations in Subsection~\ref{G23}, {\em infra}, show that, for each $\lambda\in \Lambda_0$, the regular nilpotent $\chi$ in standard Levi form gives an irreducible baby Verma module. The result then follows from Proposition~\ref{prop1} (and one can check directly that $\left\vert\Lambda_\chi/W_\bullet\right\vert=3$).	
\end{proof}

\begin{rmk}
	From the discussion in Section \ref{sec3}, it is sufficient to check the condition of Proposition~\ref{prop1} for representatives $\lambda\in\Lambda_0/{W_\bullet}$.
\end{rmk}

\begin{rmk}
	By Premet's theorem \cite{Pr1.1,Pr3.1}, Proposition~\ref{prop1} gives a proof of Humphreys' conjecture on blocks when Jantzen's standard assumptions hold. This is similar to the proof of Proposition~\ref{BlockNumb}, {\em supra}.
\end{rmk}

Proposition~\ref{prop1} shows that Humphreys' conjecture on blocks holds when irreducible baby Verma modules exist. The next proposition shows what happens when they don't.

\begin{prop}\label{NoIrred}
	Let $\lambda\in \Lambda_0$. If there does not exist $\mu_\lambda\in\fb^{\perp}$ such that $Z_{\mu_\lambda}(\lambda)$ is irreducible, then there exists $1\leq s\leq p^N-1$ such that, for all $\chi\in\fb^{\perp}$,  $Z_\chi(\lambda)$ has an $s$-dimensional submodule.
\end{prop}

\begin{proof}
	If there does not exist $\mu_\lambda\in\fb^{\perp}$ such that $Z_{\mu_\lambda}(\lambda)$ is irreducible then, using the above notation, $$\bigcap_{s=1}^{p^N-1}N_{\lambda,s}^c=\emptyset.$$ Since each $N_{\lambda,s}^c$ is open in $\fb^{\perp}$, and each non-empty open set in $\fb^{\perp}$ is dense, this implies that there exists $1\leq s\leq p^N-1$ such that $N_{\lambda,s}^c=\emptyset$. This implies that $N_{\lambda,s}=\fb^{\perp}$, as required.
\end{proof}

We end by observing an obvious generalisation of the statement that, for $\lambda\in\Lambda_0$, $S_\lambda$ is open dense in $\fb^{\perp}$ whenever there exists $\chi\in\fb^{\perp}$ with $Z_\chi(\lambda)$ irreducible.

\begin{prop}
	Let $\lambda\in \Lambda_0$. Suppose that there exists $\chi_\lambda\in\fb^{\perp}$ and $0\leq k\leq N$ such that every submodule of $Z_{\chi_\lambda}(\lambda)$ has dimension divisible by $p^k$. Then the subset $$V_\lambda\coloneqq\{\mu\in\fb^{\perp}\,\vert\,\mbox{Each } U_\mu(\fg)\mbox{-submodule of }\,Z_\mu(\lambda)\,\,\mbox{has dimension divisible by }\, p^{k}\}$$ is a dense open subset of $\fb^{\perp}$.
\end{prop}

\begin{proof}
	The result follows easily once we note that $$V_\lambda=\bigcap_{\substack{1\leq s\leq p^N \\ p^k\nmid s}}N_{\lambda,s}^c.$$
\end{proof}
\begin{rmk}
	This proposition therefore allows us to use the results of Appendix~\ref{sec6} to find dense open subsets $V_\lambda$ of $\fb^{\perp}$. These subsets are thus candidates for the sets $\widehat{C}_\lambda$ discussed earlier; all that remains to show is that $V_\lambda\subseteq C_\lambda$ for all $\lambda\in\Lambda_0$. If this were to hold, then the previous discussion would give a proof of Humphreys' conjecture on blocks for such $\fg$.
\end{rmk}

\newpage
\appendix
\section{Divisibility bounds}\label{sec6}

By Proposition~\ref{prop1}, Humphreys' conjecture on blocks holds whenever, for each $\lambda\in\Lambda_0$, there exists $\mu_\lambda\in\fb^\perp$ such that $Z_{\mu_\lambda}(\lambda)$ is an irreducible $U_{\mu_\lambda}(\fg)$-module. The natural choice for such $\mu_\lambda$ is the $\chi\in\fg^{*}$ which is regular nilpotent in standard Levi form. For such $\chi$, one way to try to show that each $Z_\chi(\lambda)$ is irreducible is to show that each $\dim(Z_\chi(\lambda))$ is divisible by $p^N$, where $N=\left\vert\Phi^{+}\right\vert$. This appendix contains some computations to determine some $k\leq N$ such that all $U_\chi(\fg)$-modules have dimension divisible by $p^k$. Unfortunately, except for the case of $G_2$ in characteristic $3$, we do not find $k$ to be equal to $N$ when assumptions (B) or (C) fail. In a few cases, we are even able to show that $p^N$ does {\em not} divide $\dim(Z_\chi(\lambda))$ for some $\lambda\in\Lambda_0$.

In this appendix, we assume $G$ is an almost-simple simply-connected algebraic group over an algebraically closed field $\bK$ of positive characteristic $p>0$, and we write $\Phi$ for its (indecomposable) root system. Specifically, let $G_{\bZ}$ be a split reductive group scheme over $\bZ$ with root data $(X(T), \Phi, \alpha\mapsto\alpha^\vee)$, let $T_\bZ$ be a split maximal torus of $G_\bZ$, and let $\fg_{\bZ}$ be the Lie ring of $G_{\bZ}$. Throughout this appendix, we think of $G$ as being obtained from $G_{\bZ}$ through base change, so $G=(G_\bZ)_{\bK}$, $T=(T_\bZ)_\bK$ and $\fg=\fg_{\bZ}\otimes_{\bZ}\bK$. In particular, the elements $e_\beta$ ($\beta\in \Phi$) and $h_\alpha$ ($\alpha\in\Pi$) form a Chevalley basis of $\fg$.

Under these assumptions, $\fg$ is a simple Lie algebra unless $\Phi$ is of type $A_n$ with $p$ dividing $n+1$; of type $B_n$, $C_n$, $D_n$, $F_4$ or $E_7$ with $p=2$; or of type $E_6$ or $G_2$ with $p=3$ (see, for example, \cite[6.4(b)]{J2.2}). If $\fg$ is simple then there exists a $G$-equivariant isomorphism $\fg\xrightarrow{\sim}\fg^{*}$ coming from the Killing form, so assumption (C) holds. We also note that assumption (A) holds for all such $G$, since $G$ equals its derived subgroup.

We consider here both those $G$ which satisfy assumption (C) and those which don't (i.e. we also consider those $G$ with $\fg$ non-simple). We focus our attention on the exceptional types $E_6$, $E_7$, $E_8$, $F_4$ and $G_2$. We generally assume throughout this appendix that $\chi$ is in standard Levi form with $I=\Pi$, although we don't make that assumption in this preliminary discussion.

When $G$ satisfies assumptions (A), (B) and (C), Premet's theorem \cite{Pr1.2,Pr3.2} (proving the second Kac-Weisfeiler conjecture \cite{KW2.2}) shows that the dimension of each $U_\chi(\fg)$-module is divisible by $p^{\dim (G\cdot\chi)/2}$. We note also that when (A) and (B) hold but (C) does not - i.e. when $\Phi=A_n$ and $p$ divides $n+1$ - Premet's theorem shows the same result for faithful irreducible $U_\chi(\fg)$-modules. When $\chi$ is regular nilpotent and assumption (C) holds, we know that $\dim (G\cdot\chi) /2=N$. Hence, in this situation we have that all irreducible $U_\chi(\fg)$-modules have dimension divisible by $p^{N}$. This means that all baby Verma modules are irreducible, and so all irreducible $U_\chi(\fg)$-modules are baby Verma modules.

Outside of the setting of Premet's theorem, there are other ways to determine powers of $p$ which divide the dimensions of all $U_\chi(\fg)$-modules. Two particular results are relevant here. Both utilize the {\bf centraliser} in $\fg$ of $\chi\in\fg^{*}$, which the reader will recall is defined as $c_{\fg}(\chi)\coloneqq \{x\in\fg\,\vert\,\chi([x,\fg])=0\}.$

The first result comes from Premet and Skryabin \cite{PS.2}, and applies when the prime $p$ is {\bf non-special} for the root system $\Phi$. This means that $p\neq 2$ when $\Phi$ is $B_n$, $C_n$ or $F_4$, and $p\neq 3$ when $\Phi=G_2$ (i.e. $p$ does not divide any non-zero off-diagonal entry of the Cartan matrix).

\begin{prop}\label{nonspec}
	Let $\chi\in\fb^{\perp}$, and let $d(\chi)\coloneqq\frac{1}{2}\codim_{\fg}(c_\fg(\chi))$. If $p$ is non-special for $\Phi$, then every $U_\chi(\fg)$-module has dimension divisible by $p^{d(\chi)}$.
\end{prop}

The second proposition we use is also due to Premet \cite{J2.2,Pr1.2,Pr2.2}. To apply it, recall that a restricted Lie algebra is called {\bf unipotent} if for all $x\in\fg$ there exists $r>0$ such that $x^{[p^r]}=0$, where $x^{[p^r]}$ denotes the image of $x$ under $r$ applications of $\,^{[p]}$. In particular, this applies to $\fn^{-}$ and any restricted subalgebras of it.

\begin{prop}\label{unip}
	Let $\chi\in\fg^{*}$. If $\fm$ is a unipotent restricted subalgebra of $\fg$ with $\chi([\fm,\fm])=0$, $\chi(\fm^{[p]})=0$ and $\fm\cap c_\fg(\chi)=0$, then every finite-dimensional $U_\chi(\fg)$-module is free over $U_\chi(\fm)$.
\end{prop}

In applying the second proposition when $\chi\in\fb^{\perp}$, the reader should note the following. Suppose $\fm$, a $\bK$-subspace of $\fg$, has a basis consisting of elements $e_{-\alpha}$ for $\alpha\in\Psi$, where $\Psi$ is some subset of $\Phi^{+}$. The condition $\chi([\fm,\fm])=0$ is clearly satisfied if $\chi(e_{-\alpha-\beta})=0$ for all $\alpha,\beta\in\Psi$. Furthermore, we have in $U(\fg)$ that $$\left(\sum_{\alpha\in\Psi}c_\alpha e_{-\alpha}\right)^{p}-\left(\sum_{\alpha\in\Psi}c_\alpha e_{-\alpha}\right)^{[p]}=\sum_{\alpha\in\Psi}c_\alpha^p (e_{-\alpha}^p-e_{-\alpha}^{[p]})=\sum_{\alpha\in\Psi}c_\alpha^p e_{-\alpha}^p$$ by the semilinearity of the map $x\mapsto x^p-x^{[p]}$, and we have $$\left(\sum_{\alpha\in\Psi}c_\alpha e_{-\alpha}\right)^{p}\in \sum_{\alpha\in\Psi}c_\alpha^p e_{-\alpha}^p + \sum_{\gamma_1,\ldots,\gamma_p\in\Psi}\bK e_{-\gamma_1-\gamma_2-\cdots-\gamma_p}$$ where we interpret $e_{-\gamma_1-\gamma_2-\cdots-\gamma_p}=0$ if $-\gamma_1-\gamma_2-\cdots-\gamma_p\notin \Phi$. We hence conclude that $$\left(\sum_{\alpha\in\Psi}c_\alpha e_{-\alpha}\right)^{[p]}\in \sum_{\gamma_1,\ldots,\gamma_p\in\Psi}\bK e_{-\gamma_1-\gamma_2-\cdots-\gamma_p}.$$ In particular, if $\chi(e_{-\gamma_1-\gamma_2-\cdots-\gamma_p})=0$ for all $\gamma_1,\ldots,\gamma_p\in\Psi$, we find that $\chi(\fm^{[p]})=0$. Furthermore, if $\Psi$ satisfies the condition that $\alpha,\beta\in\Psi$, $\alpha+\beta\in\Phi$ implies $\alpha+\beta\in\Psi$ (we call this the condition of $\Psi$ being {\bf closed}), then it is enough to check that $\chi(e_{-\alpha-\beta})=0$ for all $\alpha,\beta\in\Psi$. Finally, we observe that $\Psi$ being closed is enough to show that $\fm$ is a subalgebra. So we may obtain a corollary to Proposition~\ref{unip}:

\begin{cor}
	Let $\chi\in\fb^{\perp}$ and let $\Psi$ be a closed subset of $\Phi^{+}$. Suppose that $\chi(e_{-\alpha-\beta})=0$ for all $\alpha,\beta\in\Psi$. Furthermore, let $\fm$ be the subspace of $\fg$ with basis consisting of the $e_{-\alpha}$ with $\alpha\in\Psi$, and suppose that $\fm\cap c_\fg(\chi)=0$. Then every finite-dimensional $U_\chi(\fg)$-module has dimension divisible by $p^{\vert\Psi\vert}$.
\end{cor}

The above discussion actually shows that this corollary can be improved a bit. Given two roots $\alpha,\beta\in\Phi$, write $C_{\alpha,\beta}\coloneqq q+1$ where $q\in\bN$ is maximal for the condition that $\beta-q\alpha$ lies in $\Phi$ (so, in particular, $[e_{\alpha},e_{\beta}]=\pm C_{\alpha,\beta} e_{\alpha+\beta}$ if $\alpha+\beta\in\Phi$). Let us say that $\Psi$ is {\bf $p$-closed} if, for all $\alpha,\beta\in\Psi$ with $\alpha+\beta\in \Phi$, either $\alpha+\beta\in \Psi$ or $p$ divides $C_{\gamma,\delta}$ for all $\gamma, \delta\in\Psi$ with $\gamma+\delta=\alpha+\beta$. Then we easily obtain the following.

\begin{cor}\label{pclos}
	Let $\chi\in\fb^{\perp}$, and let $\Psi$ be a $p$-closed subset of $\Phi^{+}$. Suppose that $\chi(e_{-\alpha-\beta})=0$ for all $\alpha,\beta\in\Psi$ with $\alpha+\beta\in\Psi$. Furthermore, let $\fm$ be the subspace of $\fg$ with basis consisting of the $e_{-\alpha}$ with $\alpha\in\Psi$, and suppose that $\fm\cap c_\fg(\chi)=0$. Then every finite-dimensional $U_\chi(\fg)$-module has dimension divisible by $p^{\vert\Psi\vert}$.
\end{cor}

Let us consider a bit further the condition that $\fm\cap c_\fg(\chi)=0$. Let $x\in \fm\cap c_\fg(\chi)$. We can then write $$x=\sum_{\alpha\in\Psi}c_\alpha e_{-\alpha}.$$ The fact that $x\in c_\fg(\chi)$ means that $\chi([x,\fg])=0$. This is equivalent to the requirement that $\chi([x,e_\beta])=0$ for all $\beta\in\Phi$ and $\chi([x,h])=0$ for all $h\in\fh$.
Let $\Delta$ be the subset of $\Phi^{-}$ such that $\chi(e_{\alpha})\neq 0$ for $\alpha\in \Delta$. We then have, for $\beta\in\Phi$, that $$0=\chi([x,e_\beta])=\sum_{\substack{\gamma\in\Psi \\ \beta-\gamma\in\Delta}}c_{\gamma}\chi([e_{-\gamma},e_\beta])$$ and, for $h\in\fh$, that $$0=\chi([x,h])=\sum_{\gamma\in\Delta}c_\gamma\gamma(h)\chi(e_{-\gamma}).$$ Showing that $\fm\cap c_\fg(\chi)=0$ then involves showing that there is no non-zero solution to these equations in $c_\gamma$.
	\begin{table}
	\label{tab1}
	\begin{center}
		\bgroup
		%\def\arraystretch{1.3}
		%\caption{}
		%\vskip 3mm
		\renewcommand{\arraystretch}{2}
		\caption{Dimensions of centralisers of regular nilpotent elements and $p$-characters}
		\begin{tabular}{|| p{2em} | p{2em} | p{4em} | p{4em} ||}
			\hline
			$G$ & $p$ & $\dim\fc_\fg(e)$ & $\dim\fc_\fg(\chi)$  \\ [0.5ex]
			\hline\hline
			$E_6$ & 2 & 8 & 8 \\ [0.5ex]
			\hline
			$E_6$ & 3* & 9 & 10 \\ [0.5ex]
			\hline
			$E_7$ & 2* & 14 & 15 \\ [0.5ex]
			\hline
			$E_7$ & 3 & 9 & 9 \\ [0.5ex]
			\hline
			$E_8$ & 2 & 16 & 16 \\ [0.5ex]
			\hline
			$E_8$ & 3 & 12 &  12 \\ [0.5ex]
			\hline
			$E_8$ & 5 & 10 & 10\\ [0.5ex]
			\hline
			$F_4$ & 2* & 8 & 6  \\ [0.5ex]
			\hline
			$F_4$ & 3 & 6 & 6 \\ [0.5ex]
			\hline
			$G_2$ & 2 & 4 & 4 \\ [0.5ex]
			\hline
			$G_2$ & 3* & 3 & 2 \\ [0.5ex]
			\hline
		\end{tabular}
		\egroup
	\end{center}
\end{table}

We now turn to the application of these propositions. In each case, we take $\chi$ to be regular nilpotent in standard Levi form and we apply one of the propositions or its corollaries to determine a divisibility bound for the dimensions of $U_\chi(\fg)$-modules. We do this for $\Phi$ of exceptional type. Principally, we compute the centraliser $\fc_\fg(\chi)$ and use its description to determine the bound. For $\Phi=G_2$ we give the explicit computations, but for the larger rank examples the results were obtained using Sage \cite{S.2}. Because of this, when there is a choice we take the structure coefficients to be as used in the Sage class {\fontfamily{pcr}\selectfont LieAlgebraChevalleyBasis\_with\_category}. However, we use the labelling of the simple roots as given in \cite{Sp.2}.

\begin{rmk}
	Our computations of $\dim\fc_\fg(\chi)$ can be compared with the computations of $\dim\fc_\fg(e)$ for $e=\sum_{\alpha\in\Pi}e_\alpha$ which can be deduced from \cite[Cor. 2.5, Thm. 2.6]{Sp.2}. The results are listed in Table~\ref{tab1}. 
	
	 When $\fg$ is simple, $\chi$ and $e$ are identified through the $G$-equivariant isomorphism $\fg\xrightarrow{\sim}\fg^{*}$, and thus $\fc_\fg(\chi)=\fc_\fg(e)$. In the subsections below, we nonetheless include calculations of $\fc_\fg(\chi)$ for the bad primes for which $\fg$ is simple, since we give explicit bases for the centralisers in these case and in some instances we use such bases to show the reducibility of the corresponding baby Verma modules.
	 
	 In the other cases (which we label with an asterisk (*) in Table~\ref{ta1}), however, we find that the dimensions of $\fc_\fg(e)$ and $\fc_\fg(\chi)$ differ from each other. Note also that we give in Table~\ref{ta1} the dimension of $\fc_\fg(\chi)$ for $G_2$ in characteristic 3, even though we do not give it in Subsection~\ref{G23} below, because it is easy to compute.
\end{rmk}

\begin{rmk}
	In our discussion of $\fg$ so far, the Lie algebra $\fg$ of $G$ has been obtained as $\fg=\fg_\bZ\otimes_\bZ \bK$, where $\fg_\bZ$ is a $\bZ$-form of the complex simple Lie algebra $\fg_\bC$. In particular, $\fg_{\bZ}$ is the $\bZ$-form coming from the chosen Chevalley basis of $\fg_\bC$, which is what gives our Chevalley basis of $\fg$. We may then also define $\fg_{\bF_p}=\fg_\bZ\otimes_{\bZ} \bF_p$, so that $\fg=\fg_{\bF_p}\otimes_{\bF_p}\bK$. Therefore, if $\chi_{\bF_p}:\fg_{\bF_p}\to\bF_p$ is a linear form, we may define $\chi:\fg\to\bK$ by linear extension. It is clear that any $\chi$ in standard Levi form may be obtained in this way. Our calculations in Sage are calculations with $\fg_{\bF_p}$ and $\chi_{\bF_p}$ rather than $\fg$. However, when $\chi$ is obtained through scalar extension from an $\bF_p$-linear form, the above discussion shows that determining the elements of $\fg$ which lie in $\fc_\fg(\chi)$ comes down to finding solutions to certain linear equations with coefficients in $\bF_p$. This in particular shows that $\fc_{\fg_{\bF_p}}(\chi_{\bF_p})\otimes_{\bF_p}\bK=\fc_{\fg}(\chi)$, so our calculations over $\bF_p$ also lead to the results over $\bK$.
\end{rmk}

\subsection{$G_2$ in characteristic 2}\label{G22}

Suppose $\Phi=G_2$ and $p=2$. Since $p$ is non-special in this case, we may apply Proposition~\ref{nonspec}. Let us therefore compute $c_\fg(\chi)$. Set $x\in\fg$ be written as $x=\sum_{\gamma\in\Phi}c_\gamma e_\gamma + \sum_{\gamma\in\Pi}d_\gamma h_\gamma$, with the $c_\gamma$, $d_\gamma$ lying in $\bK$. Then the relations required for $x\in c_{\fg}(\chi)$ are as follows:

\begin{itemize}
	\item $0=\chi([x,e_{3\alpha+2\beta}])=0$,
	\item $0=\chi([x,e_{3\alpha+\beta}])=c_{-3\alpha-2\beta}\chi([e_{-3\alpha-2\beta},e_{3\alpha+\beta}])=c_{-3\alpha-2\beta}\chi(e_{-\beta})= c_{-3\alpha-2\beta}$,
	\item $0=\chi([x,e_{2\alpha+\beta}])=c_{-3\alpha-\beta}\chi([e_{-3\alpha-\beta},e_{2\alpha+\beta}])=c_{-3\alpha-\beta}\chi(e_{-\alpha})= c_{-3\alpha-\beta},$
	\item $0=\chi([x,e_{\alpha+\beta}])=c_{-2\alpha-\beta}\chi([e_{-2\alpha-\beta},e_{\alpha+\beta}])=c_{-2\alpha-\beta}\chi(2e_{-\alpha})=0$,
	\item $0=\chi([x,e_{\beta}])=c_{-\alpha-\beta}\chi([e_{-\alpha-\beta},e_{\beta}])=c_{-\alpha-\beta}\chi(e_{-\alpha})= c_{-\alpha-\beta}$,
	\item $0=\chi([x,e_{\alpha}])=c_{-\alpha-\beta}\chi([e_{-\alpha-\beta},e_{\alpha}])=c_{-\alpha-\beta}\chi(-3e_{-\alpha})=-3c_{-\alpha-\beta}=c_{-\alpha-\beta}$,
	\item $0=\chi([x,h_\alpha])=c_{-\alpha}\chi(\alpha(h_\alpha)e_{-\alpha}) + c_{-\beta}\chi(\beta(h_\alpha)e_{-\beta})=2c_{-\alpha}-3c_{-\beta}=c_{-\beta}$,
	\item $0=\chi([x,h_\beta])=c_{-\alpha}\chi(\alpha(h_\beta)e_{-\alpha}) + c_{-\beta}\chi(\beta(h_\beta)e_{-\beta})=-c_{-\alpha} +2c_{-\beta}=c_{-\alpha}$,
	\item $0=\chi([x,e_{-\alpha}])=d_{\alpha}\chi([h_{\alpha},e_{-\alpha}]) + d_{\beta}\chi([h_{\beta},e_{-\alpha}])=d_{\alpha}\chi(-\alpha(h_\alpha)e_{-\alpha}) + d_{\beta}\chi(-\alpha(h_{\beta})e_{-\alpha})=-2d_{\alpha}+d_{\beta}=d_{\beta}$,
	\item $0=\chi([x,e_{-\beta}])=d_{\alpha}\chi([h_{\alpha},e_{-\beta}]) + d_{\beta}\chi([h_{\beta},e_{-\beta}])=d_{\alpha}\chi(-\beta(h_\alpha)e_{-\beta}) + d_{\beta}\chi(-\beta(h_{\beta})e_{-\beta})=3d_{\alpha}-2d_{\beta}=d_{\alpha}$,
	\item $0=\chi([x,e_{-\alpha-\beta}])=c_{\alpha}\chi([e_{\alpha},e_{-\alpha-\beta}]) + c_{\beta}\chi([e_{\beta},e_{-\alpha-\beta}])= 3c_{\alpha}\chi(e_{-\beta})- c_{\beta}\chi(e_{-\alpha})=c_{\alpha}+c_{\beta}$,
	\item $0=\chi([x,e_{-2\alpha-\beta}])=c_{\alpha+\beta}\chi([e_{\alpha+\beta},e_{-2\alpha-\beta}])=- 2c_{\alpha+\beta}\chi(e_{-\alpha})=0$,
	\item $0=\chi([x,e_{-3\alpha-\beta}])=c_{2\alpha+\beta}\chi([e_{2\alpha+\beta},e_{-3\alpha-\beta}])=- c_{2\alpha+\beta}\chi(e_{-\alpha})= c_{2\alpha+\beta}$,
	\item $0=\chi([x,e_{-3\alpha-2\beta}])=c_{3\alpha+\beta}\chi([e_{3\alpha+\beta},e_{-3\alpha-2\beta}])=- c_{3\alpha+\beta}\chi(e_{-\alpha})= c_{3\alpha+\beta}$.
\end{itemize}

We therefore conclude that $$c_{\fg}(\chi)=\{ae_{-2\alpha-\beta} + b(e_{\alpha}+e_{\beta}) + ce_{\alpha+\beta} + de_{3\alpha+2\beta}\,\vert\,a,b,c,d\in\bK\}$$ and so is 4-dimensional. Hence, $d(\chi)=\frac{1}{2}(14-4)=5$, and so by Proposition~\ref{nonspec} we conclude that every finite-dimensional $U_\chi(\fg)$-module has dimension divisible by $2^5$.

We furthermore note that a $U_\chi(\fg)$-module of dimension $2^5$ does indeed exist in this case. Let $\lambda\in \Lambda_0$ be such that $\lambda(h_\beta)=0$, and let us write $\omega_1\in \Lambda_0$ for the map with $\omega_1(h_\alpha)=1$ and $\omega_1(h_\beta)=0$. We may then define a $U_\chi(\fg)$-module homomorphism $$Z_\chi(\lambda)\to Z_\chi(\lambda-\omega_1),\qquad v_{\lambda}\mapsto e_{-2\alpha-\beta}v_{\lambda-\omega_1}.$$
This has a kernel of dimension $2^5$ and so both the kernel and image of this homomorphism are $U_\chi(\fg)$-modules of dimension $2^5$. 

\subsection{$G_2$ in characteristic 3}\label{G23}

Suppose $\Phi=G_2$ and $p=3$. Note that this Lie algebra is not simple, since it has an ideal generated by the short roots. In this case $p$ is not non-special for $\Phi$ so we cannot apply Proposition~\ref{nonspec}. Instead, we want to apply Proposition~\ref{unip}, and so we need to find an appropriate $\fm$. Take $\fm=\fn^{-}$. In this case, $\Psi=\Phi^{+}$ is closed and $\chi(e_{-\gamma-\delta})=0$ for all $\gamma,\delta\in \Psi$. In the notation of the previous discussion, we have $\Delta=\{-\alpha,-\beta\}$.

Let $x=\sum_{\gamma\in\Phi^{+}} c_\alpha e_{-\alpha}$. %$x=c_{\alpha}e_{-\alpha}+c_{\beta}e_{-\beta}+c_{\alpha+\beta}e_{-\alpha-\beta}+c_{2\alpha+\beta}e_{-2\alpha-\beta}+c_{3\alpha+\beta}e_{-3\alpha-\beta}+c_{3\alpha+2\beta}e_{-3\alpha-2\beta}$. 
Then the relations required for $x\in c_{\fg}(\chi)$ are as follows:
\begin{itemize}
	\item $0=\chi([x,e_{3\alpha+2\beta}])=0$,
	\item $0=\chi([x,e_{3\alpha+\beta}])=c_{3\alpha+2\beta}\chi([e_{-3\alpha-2\beta},e_{3\alpha+\beta}])=c_{3\alpha+2\beta}\chi(e_{-\beta})= c_{3\alpha+2\beta}$,
	\item $0=\chi([x,e_{2\alpha+\beta}])=c_{3\alpha+\beta}\chi([e_{-3\alpha-\beta},e_{2\alpha+\beta}])=c_{3\alpha+\beta}\chi(e_{-\alpha})= c_{3\alpha+\beta}$,
	\item $0=\chi([x,e_{\alpha+\beta}])=c_{2\alpha+\beta}\chi([e_{-2\alpha-\beta},e_{\alpha+\beta}])=c_{2\alpha+\beta}\chi(2e_{-\alpha})= 2c_{2\alpha+\beta}$,
	\item $0=\chi([x,e_{\beta}])=c_{\alpha+\beta}\chi([e_{-\alpha-\beta},e_{\beta}])=c_{\alpha+\beta}\chi(e_{-\alpha})=c_{\alpha+\beta}$,
	\item $0=\chi([x,e_{\alpha}])=c_{\alpha+\beta}\chi([e_{-\alpha-\beta},e_{\alpha}])=c_{\alpha+\beta}\chi(-3e_{-\alpha})=0$,
	\item $0=\chi([x,h_\alpha])=c_{\alpha}\chi(\alpha(h_\alpha)e_{-\alpha}) + c_{\beta}\chi(\beta(h_\alpha)e_{-\beta})=2c_\alpha -3c_\beta=2c_\alpha$,
	\item $0=\chi([x,h_\beta])=c_{\alpha}\chi(\alpha(h_\beta)e_{-\alpha}) + c_{\beta}\chi(\beta(h_\beta)e_{-\beta})=-c_\alpha +2c_\beta$.
\end{itemize}

It is easy to see that these relations force $x=0$, so $\fm\cap c_{\fg}(\chi)=0$. Hence, Proposition~\ref{unip} shows that every finite-dimensional $U_\chi(\fg)$-module has dimension divisible by $3^6$, which is $3^{\dim\fn^{-}}$. So in this case each baby Verma module $Z_\chi(\lambda)$ is irreducible.

\subsection{$F_4$ in characteristic 2}\label{F42}

Set $\Phi=F_4$ and $p=2$. Since $p$ is not non-special in this case we need to use Proposition~\ref{unip}; in fact, we use Corollary~\ref{pclos}. Set $\fm$ to be the subspace of $\fn^{-}$ with basis given by the elements $e_{-\alpha}$ for $\alpha\in\Psi\coloneqq\Phi^{+}\setminus\{\alpha_2+ 2\alpha_3\}$. It is straightforward to see that $\Psi$ is 2-closed. We want to see that $\fm\cap\fc_\fg(\chi)=0$. We do this by giving a basis of $\fc_\fg(\chi)$ as follows:

\begin{enumerate}
	\item $e_{-\alpha_2-\alpha_3-\alpha_4} + e_{-\alpha_2-2\alpha_3}$;
	\item $e_{\alpha_3} +e_{\alpha_4}$;
	\item $e_{\alpha_3+\alpha_4}$;
	\item $e_{\alpha_1+\alpha_2+2\alpha_3+\alpha_4} + e_{\alpha_2+2\alpha_3+2\alpha_4}$;
	\item $e_{\alpha_1+2\alpha_2+3\alpha_3+\alpha_4}+e_{\alpha_1+2\alpha_2+2\alpha_3+2\alpha_4}$;
	\item $e_{2\alpha_1+3\alpha_2+4\alpha_3+2\alpha_4}$.
\end{enumerate}

It is clear from this basis description that $\fc_\fg(\chi)\cap\fm=0$. Hence, Corollary~\ref{pclos} applies and we get that every finite-dimensional $U_\chi(\fg)$-module has dimension divisible by $2^{\left\vert\Psi\right\vert}=2^{\left\vert\Phi^{+}\right\vert-1}=2^{23}$.

Now, set $\fr$ to be the $\bK$-subspace of $\fg$ generated by $e_\beta$ for all $\beta\in\Phi^{+}\setminus\{\alpha_3,\alpha_4\}$, by $h_1, h_2+h_3, h_3+h_4$, by $e_{\alpha_3}+e_{\alpha_4}$, and by $e_{-\alpha_2-\alpha_3-\alpha_4}+e_{-\alpha_2-2\alpha_3}, e_{-\alpha_3-\alpha_4}$ and $e_{-\alpha_3}+e_{-\alpha_4}$. This has dimension 29. One may check that it is in fact a subalgebra of $\fg$ (using that the characteristic of $\bK$ is 2). One may also check that $\chi([\fr,\fr])=0$, and that $\chi(\fr^{[p]})=0$ (since the characteristic is 2, we have $(x+y)^{[2]}=[x,y]$ whenever $x^{[2]}=y^{[2]}=0$).

Then $U_\chi(\fr)$ has dimension $2^{29}$ and has a 1-dimensional trivial module $\bK_\chi$. Therefore, $U_\chi(\fg)\otimes_{U_\chi(\fr)}\bK_\chi$ is a $U_\chi(\fg)$-module of dimension $2^{23}$, so the divisibility bound we found is strict.

\subsection{$F_4$ in characteristic 3}\label{F43}

Set $\Phi=F_4$ and $p=3$. Since $p$ is non-special in this case, we may apply Proposition~\ref{nonspec}. We must therefore give $c_\fg(\chi)$, and Sage computations show that $\fc_\fg(\chi)$ is the $\bK$-subspace of $\fg$ with the following basis:

\begin{enumerate}
	\item $e_{-\alpha_2-2\alpha_3-\alpha_4} + 2e_{-\alpha_1-\alpha_2-\alpha_3-\alpha_4} + e_{-\alpha_1-\alpha_2-2\alpha_3}$;
	\item $2e_{\alpha_1}+2e_{\alpha_2}+e_{\alpha_3} +e_{\alpha_4}$;
	\item $e_{\alpha_2+\alpha_3+\alpha_4}+2e_{\alpha_2+2\alpha_3} + 2e_{\alpha_1+\alpha_2+\alpha_3}$,
	\item $e_{\alpha_2+2\alpha_3+2\alpha_4} + e_{\alpha_1+2\alpha_2+2\alpha_3} + e_{\alpha_1+\alpha_2+2\alpha_3+\alpha_4}$;
	\item $e_{\alpha_1+2\alpha_2+3\alpha_3+\alpha_4}+2e_{\alpha_1+2\alpha_2+2\alpha_3+2\alpha_4}$;
	\item $e_{2\alpha_1+3\alpha_2+4\alpha_3+2\alpha_4}$.
	
\end{enumerate}

Therefore, $\dim c_{\fg}(\chi)=6$, and so $d(\chi)=23=\left\vert\Phi^{+}\right\vert-1$. Hence, every finite-dimensional $U_\chi(\fg)$-module has dimension divisible by $3^{23}$.

\subsection{$E_6$ in characteristic 2}\label{E62}
Suppose $\Phi=E_6$ and $p=2$. Since $p$ is non-special in this case, we may apply Proposition~\ref{nonspec}. We must therefore give $c_\fg(\chi)$, and Sage computations show that $\fc_\fg(\chi)$ is the $\bK$-subspace of $\fg$ with the following basis:

\begin{enumerate}
	\item $e_{-\alpha_2-\alpha_3-\alpha_4}+e_{-\alpha_2-\alpha_3-\alpha_6} + e_{-\alpha_3-\alpha_4-\alpha_6}$;
	\item $e_{\alpha_1}+e_{\alpha_2}+e_{\alpha_3} +e_{\alpha_4} + e_{\alpha_5} +e_{\alpha_6}$;
	\item $e_{\alpha_1+\alpha_2}+e_{\alpha_2+\alpha_3} + e_{\alpha_3+\alpha_4} + e_{\alpha_3+\alpha_6} + e_{\alpha_4+\alpha_5}$;
	\item $e_{\alpha_1+\alpha_2+\alpha_3+\alpha_4} + e_{\alpha_2+\alpha_3+\alpha_4+\alpha_5} + e_{\alpha_1+\alpha_2+\alpha_3+\alpha_6}+e_{\alpha_3+\alpha_4+\alpha_5+\alpha_6}$;
	\item $e_{\alpha_1+\alpha_2+\alpha_3+\alpha_4+\alpha_6}+e_{\alpha_2+2\alpha_3+\alpha_4+\alpha_6} + e_{\alpha_2+\alpha_3+\alpha_4+\alpha_5+\alpha_6}$;
	\item $e_{\alpha_1+2\alpha_2+2\alpha_3+\alpha_4+\alpha_6}+e_{\alpha_1+\alpha_2+2\alpha_3+\alpha_4+\alpha_5+\alpha_6} + e_{\alpha_2+2\alpha_3+2\alpha_4+\alpha_5+\alpha_6}$;
	\item $e_{\alpha_1+2\alpha_2+2\alpha_3+\alpha_4+\alpha_5+\alpha_6} + e_{\alpha_1+\alpha_2+2\alpha_3+2\alpha_4 +\alpha_5 +\alpha_6}$;
	\item $e_{\alpha_1+2\alpha_2+3\alpha_3+2\alpha_4+\alpha_5+2\alpha_6}$.
\end{enumerate}

In particular we see that $\dim c_{\fg}(\chi)=8$, and so $d(\chi)=35=\left\vert\Phi^{+}\right\vert-1$. Hence, every finite-dimensional $U_\chi(\fg)$-module has dimension divisible by $2^{35}$.

\subsection{$E_6$ in characteristic 3}\label{E63}

Suppose $\Phi=E_6$ and $p=3$. Since $p$ is non-special in this case, we may apply Proposition~\ref{nonspec}. We must therefore give $c_\fg(\chi)$, and Sage computations show that $\fc_\fg(\chi)$ is the $\bK$-subspace of $\fg$ with the following basis:
\begin{enumerate}
	\item $e_{-\alpha_2-\alpha_3-\alpha_4-\alpha_5}+2e_{-\alpha_3-\alpha_4-\alpha_5-\alpha_6} + e_{-\alpha_2-\alpha_3-\alpha_4-\alpha_6}  + 2e_{-\alpha_1-\alpha_2-\alpha_3-\alpha_4}+e_{-\alpha_1-\alpha_2-\alpha_3-\alpha_6}$;
	\item $2e_{-\alpha_1}+e_{-\alpha_2}+2e_{-\alpha_4}+e_{-\alpha_5}$;
	\item $h_1+2h_2+h_4+2h_5$;
	\item $e_{\alpha_1}+e_{\alpha_2}+e_{\alpha_3} +e_{\alpha_4} + e_{\alpha_5} +e_{\alpha_6}$;
	\item $e_{\alpha_1+\alpha_2+\alpha_3}+e_{\alpha_2+\alpha_3+\alpha_4} + e_{\alpha_3+\alpha_4+\alpha_5} + e_{\alpha_2+\alpha_3+\alpha_6} + e_{\alpha_3+\alpha_4+\alpha_6}$;
	\item $e_{\alpha_1+\alpha_2+\alpha_3+\alpha_4} + e_{\alpha_2+\alpha_3+\alpha_4+\alpha_5} + 2e_{\alpha_1+\alpha_2+\alpha_3+\alpha_6}+2e_{\alpha_3+\alpha_4+\alpha_5+\alpha_6}$;
	\item $2e_{\alpha_1+\alpha_2+\alpha_3+\alpha_4+\alpha_6}+2e_{\alpha_2+2\alpha_3+\alpha_4+\alpha_6} + e_{\alpha_2+\alpha_3+\alpha_4+\alpha_5+\alpha_6}+2e_{\alpha_1+\alpha_2+\alpha_3+\alpha_4+\alpha_5}$;
	\item $2e_{\alpha_1+2\alpha_2+2\alpha_3+\alpha_4+\alpha_6}+e_{\alpha_1+\alpha_2+2\alpha_3+\alpha_4+\alpha_5+\alpha_6} + e_{\alpha_2+2\alpha_3+2\alpha_4+\alpha_5+\alpha_6}$;
	\item $e_{\alpha_1+2\alpha_2+2\alpha_3+\alpha_4+\alpha_5+\alpha_6} + 2e_{\alpha_1+\alpha_2+2\alpha_3+2\alpha_4 +\alpha_5 +\alpha_6}$;
	\item $e_{\alpha_1+2\alpha_2+3\alpha_3+2\alpha_4+\alpha_5+2\alpha_6}$.
\end{enumerate}

Hence, $\dim(c_\fg(\chi))=10$ and so $d(\chi)=34=\left\vert\Phi^{+}\right\vert-2$. Proposition~\ref{nonspec} then says that all $U_\chi(\fg)$-modules have dimension divisible by $3^{34}$.

Now, set $\fr$ be the $\bK$-subspace of $\fg$ generated by $e_\beta$ for all $\beta\in\Phi^{+}$, by $h_1, h_2, h_4, h_5$ and $h_6$, and by $2e_{-\alpha_1}+e_{-\alpha_2}$ and $2e_{-\alpha_4}+e_{-\alpha_5}$. This has dimension 43. One may check that it is in fact a subalgebra of $\fg$ (using that the characteristic is 3). One may also check that $\chi([\fr,\fr])=0$ and that $\chi(\fr^{[p]})=0$.

Then $U_\chi(\fr)$ has dimension $3^{43}$ and has a 1-dimensional module $\bK_\chi$. Therefore $U_\chi(\fg)\otimes_{U_\chi(\fr)}\bK_\chi$ is a $U_\chi(\fg)$-module of dimension $3^{35}$. In particular, this shows that it is {\em not} true in this case that all baby Verma modules are irreducible $U_\chi(\fg)$-modules. It obviously, however, does not imply that our divisibility bound is strict.

\subsection{$E_7$ in characteristic 2}\label{E72}

Suppose $\Phi=E_7$ and $p=2$. Since $p$ is non-special in this case, we may apply Proposition~\ref{nonspec}. We must therefore give $c_\fg(\chi)$, and Sage computations show that $\fc_\fg(\chi)$ is the $\bK$-subspace of $\fg$ with the following basis:
\begin{enumerate}
	\item $e_{-\alpha_1-2\alpha_2-2\alpha_3-2\alpha_4-\alpha_5-\alpha_7} + e_{-\alpha_2-2\alpha_3-2\alpha_4-2\alpha_5-\alpha_6-\alpha_7} + e_{-\alpha_1-\alpha_2-2\alpha_3-2\alpha_4-\alpha_5-\alpha_6-\alpha_7} \\ + e_{-\alpha_1-\alpha_2-\alpha_3-2\alpha_4-2\alpha_5-\alpha_6-\alpha_7}$;
	\item $e_{-\alpha_1-\alpha_2-\alpha_3-\alpha_4-\alpha_5}+e_{-\alpha_3-2\alpha_4-\alpha_5-\alpha_7} + e_{-\alpha_1-\alpha_2-\alpha_3-\alpha_4-\alpha_7} + e_{-\alpha_2-\alpha_3-\alpha_4-\alpha_5-\alpha_7};$
	\item $e_{-\alpha_3-\alpha_4-\alpha_5} + e_{-\alpha_4-\alpha_5-\alpha_7}+e_{-\alpha_3-\alpha_4-\alpha_7}$;
	\item $e_{-\alpha_1}+e_{-\alpha_3}+e_{-\alpha_7}$;
	\item $h_1+h_3+h_7$;
	\item $e_{\alpha_1}+e_{\alpha_2}+e_{\alpha_3} +e_{\alpha_4} + e_{\alpha_5} +e_{\alpha_6}+e_{\alpha_7}$;
	\item $e_{\alpha_1+\alpha_2}+e_{\alpha_2+\alpha_3} + e_{\alpha_3+\alpha_4} + e_{\alpha_4+\alpha_5} + e_{\alpha_4+\alpha_7}+e_{\alpha_5+\alpha_6}$;
	\item $e_{\alpha_1+\alpha_2+\alpha_3+\alpha_4} + e_{\alpha_2+\alpha_3+\alpha_4+\alpha_5} + e_{\alpha_2+\alpha_3+\alpha_4+\alpha_7} 
	+e_{\alpha_4+\alpha_5+\alpha_6+\alpha_7}+e_{\alpha_3+\alpha_4+\alpha_5+\alpha_6}$;
	\item $e_{\alpha_1+\alpha_2+\alpha_3+\alpha_4+\alpha_7}+e_{\alpha_2+\alpha_3+\alpha_4+\alpha_5+\alpha_7} + e_{\alpha_3+\alpha_4+\alpha_5+\alpha_6+\alpha_7}+e_{\alpha_3+2\alpha_4+\alpha_5+\alpha_7}$;
	\item $e_{\alpha_1+\alpha_2+\alpha_3+2\alpha_4+\alpha_5+\alpha_7}+e_{\alpha_2+2\alpha_3+2\alpha_4+\alpha_5+\alpha_7} + e_{\alpha_2+\alpha_3+2\alpha_4+\alpha_5+\alpha_6+\alpha_7} + e_{\alpha_3+2\alpha_4+2\alpha_5+\alpha_6+\alpha_7}$;
	\item $e_{\alpha_1+\alpha_2+\alpha_3+2\alpha_4+\alpha_5+\alpha_6+\alpha_7} + e_{\alpha_2+2\alpha_3+2\alpha_4 +\alpha_5 +\alpha_6+\alpha_7} +e_{\alpha_2+\alpha_3+2\alpha_4+2\alpha_5+\alpha_6+\alpha_7}$;
	\item $e_{\alpha_1+2\alpha_2+2\alpha_3+2\alpha_4+\alpha_5+\alpha_7}+e_{\alpha_1+\alpha_2+2\alpha_3+2\alpha_4+\alpha_5+\alpha_6+\alpha_7} + e_{\alpha_1+\alpha_2+\alpha_3+2\alpha_4+2\alpha_5+\alpha_6+\alpha_7}$;
	\item $e_{\alpha_1+2\alpha_2+2\alpha_3+2\alpha_4+2\alpha_5+\alpha_6+\alpha_7}+e_{\alpha_1+\alpha_2+2\alpha_3+3\alpha_4+2\alpha_5+\alpha_6+\alpha_7}  + e_{\alpha_2+2\alpha_3+3\alpha_4+2\alpha_5+\alpha_6+2\alpha_7}$;
	\item $e_{\alpha_1+2\alpha_2+3\alpha_3+3\alpha_3+2\alpha_5+\alpha_6+\alpha_7} + e_{\alpha_1+2\alpha_2+2\alpha_3+3\alpha_4+2\alpha_5+\alpha_6+2\alpha_7}$;
	\item $e_{\alpha_1+2\alpha_2+3\alpha_3+4\alpha_4+3\alpha_5+2\alpha_6+2\alpha_7}$.
\end{enumerate}
We conclude that $\dim(\fc_\fg(\chi))=15$ and so $d(\chi)=59=\left\vert\Phi^{+}\right\vert-4$. We then conclude from Proposition~\ref{nonspec} that every finite-dimensional $U_\chi(\fg)$-module has dimension divisible by $2^{59}$.

\subsection{$E_7$ in characteristic 3}\label{E73}

Suppose $\Phi=E_7$ and $p=3$. Since $p$ is non-special in this case, we may apply Proposition~\ref{nonspec}. We must therefore give $c_\fg(\chi)$, and Sage computations show that $\fc_\fg(\chi)$ is the $\bK$-subspace of $\fg$ with the following basis:

\begin{enumerate}
	\item $e_{-\alpha_2-\alpha_3-\alpha_4-\alpha_5}+2e_{-\alpha_2-\alpha_3-\alpha_4-\alpha_7} + 2e_{-\alpha_3-\alpha_4-\alpha_5-\alpha_6}
	+e_{-\alpha_3-\alpha_4-\alpha_5-\alpha_7}+e_{-\alpha_4-\alpha_5-\alpha_6-\alpha_7}$;
	\item $e_{\alpha_1}+e_{\alpha_2}+e_{\alpha_3} +e_{\alpha_4} + e_{\alpha_5} +e_{\alpha_6}+e_{\alpha_7}$;
	\item $e_{\alpha_1+\alpha_2+\alpha_3}+e_{\alpha_2+\alpha_3+\alpha_4} + e_{\alpha_3+\alpha_4+\alpha_5} + e_{\alpha_3+\alpha_4+\alpha_7} + e_{\alpha_4+\alpha_5+\alpha_6}+e_{\alpha_4+\alpha_5+\alpha_7}$;
	\item $e_{\alpha_1+\alpha_2+\alpha_3+\alpha_4+\alpha_5} + e_{\alpha_2+\alpha_3+\alpha_4+\alpha_5+\alpha_6} + 2e_{\alpha_2+\alpha_3+\alpha_4+\alpha_5+\alpha_7}
	+e_{\alpha_3+2\alpha_4+\alpha_5+\alpha_7}+e_{\alpha_3+\alpha_4+\alpha_5+\alpha_6+\alpha_7}$;
	\item $2e_{\alpha_1+\alpha_2+\alpha_3+2\alpha_4+\alpha_5+\alpha_7}+e_{\alpha_1+\alpha_2+\alpha_3+\alpha_4+\alpha_5+\alpha_6+\alpha_7}  + e_{\alpha_2+2\alpha_3+2\alpha_4+\alpha_5+\alpha_7}+e_{\alpha_2+\alpha_3+2\alpha_4+\alpha_5+\alpha_6+\alpha_7}\\+2e_{\alpha_3+2\alpha_4+2\alpha_5+\alpha_6+\alpha_7}$;
	\item $e_{\alpha_1+2\alpha_2+2\alpha_3+2\alpha_4+\alpha_5+\alpha_7}+e_{\alpha_1+\alpha_2+2\alpha_3+2\alpha_4+\alpha_5+\alpha_6+\alpha_7} + 2e_{\alpha_1+\alpha_2+\alpha_3+2\alpha_4+2\alpha_5+\alpha_6+\alpha_7}$;
	\item $e_{\alpha_1+2\alpha_2+2\alpha_3+2\alpha_4+2\alpha_5+\alpha_6+\alpha_7} + 2e_{\alpha_1+\alpha_2+2\alpha_3+3\alpha_4 +2\alpha_5 +\alpha_6+\alpha_7} +e_{\alpha_2+2\alpha_3+3\alpha_4+2\alpha_5+\alpha_6+2\alpha_7}$;
	\item $e_{\alpha_1+2\alpha_2+3\alpha_3+3\alpha_4+2\alpha_5+\alpha_6+\alpha_7}+2e_{\alpha_1+2\alpha_2+2\alpha_3+3\alpha_4+2\alpha_5+\alpha_6+2\alpha_7}$;
	\item $e_{\alpha_1+2\alpha_2+3\alpha_3+4\alpha_4+3\alpha_5+2\alpha_6+2\alpha_7}$.
\end{enumerate}

In particular we see that $\dim c_{\fg}(\chi)=9$, and so $d(\chi)=62=\left\vert\Phi^{+}\right\vert-1$. Every finite-dimensional $U_\chi(\fg)$-module therefore has dimension divisible by $3^{62}$.

\subsection{$E_8$ in characteristic 2}\label{E82}

Suppose $\Phi=E_8$ and $p=2$. Since $p$ is non-special in this case, we may apply Proposition~\ref{nonspec}. We must therefore give $c_\fg(\chi)$, and Sage computations show that $\fc_\fg(\chi)$ is the $\bK$-subspace of $\fg$ with the following basis:
\begin{enumerate}
	\item $e_{-\alpha_2-2\alpha_3-3\alpha_4-4\alpha_5 - 2\alpha_6-\alpha_7-2\alpha_8}+e_{-\alpha_1-\alpha_2-2\alpha_3-3\alpha_4 -3\alpha_5-2\alpha_6-\alpha_7-2\alpha_8} + e_{-\alpha_1-2\alpha_2-2\alpha_3-3\alpha_4-3\alpha_5-2\alpha_6-\alpha_7-\alpha_8} \\ +e_{-\alpha_1-2\alpha_2-2\alpha_3-2\alpha_4-3\alpha_5-2\alpha_6-\alpha_7-2\alpha_8}$;
	\item $e_{-\alpha_2-2\alpha_3-2\alpha_4-2\alpha_5-\alpha_6-\alpha_8} + e_{-\alpha_2-\alpha_3-\alpha_4-2\alpha_5-2\alpha_6-\alpha_7-\alpha_8} + e_{-\alpha_2-\alpha_3-2\alpha_4-2\alpha_5-\alpha_6-\alpha_7-\alpha_8}
	\\ + e_{-\alpha_3-2\alpha_4-2\alpha_5-2\alpha_6-\alpha_7-\alpha_8}$;
	\item $e_{-\alpha_2-\alpha_3-\alpha_4-\alpha_5-\alpha_6}+e_{-\alpha_2-\alpha_3-\alpha_4-\alpha_5-\alpha_8} + e_{-\alpha_4-2\alpha_5-\alpha_6-\alpha_8} + e_{-\alpha_3-\alpha_4-\alpha_5-\alpha_6-\alpha_8}$;
	\item $e_{-\alpha_4-\alpha_5-\alpha_6} + e_{-\alpha_4-\alpha_5-\alpha_8} + e_{-\alpha_5-\alpha_6-\alpha_8}$;
	\item $e_{\alpha_1}+e_{\alpha_2}+e_{\alpha_3} +e_{\alpha_4} + e_{\alpha_5} +e_{\alpha_6}+e_{\alpha_7}+e_{\alpha_8}$;
	\item $e_{\alpha_1+\alpha_2}+e_{\alpha_2+\alpha_3} + e_{\alpha_3+\alpha_4} + e_{\alpha_4+\alpha_5} + e_{\alpha_5+\alpha_6}+e_{\alpha_5+\alpha_8} + e_{\alpha_6+\alpha_7} $
	\item $e_{\alpha_1+\alpha_2+\alpha_3+\alpha_4} + e_{\alpha_2+\alpha_3+\alpha_4+\alpha_5} + e_{\alpha_3+\alpha_4+\alpha_5+\alpha_6} +e_{\alpha_3+\alpha_4+\alpha_5+\alpha_8}+e_{\alpha_5+\alpha_6+\alpha_7+\alpha_8} + e_{\alpha_4+\alpha_5+\alpha_6+\alpha_7}$;
	\item $e_{\alpha_1+\alpha_2+\alpha_3+\alpha_4+\alpha_5+\alpha_6+\alpha_8}+e_{\alpha_2+\alpha_3+\alpha_4+2\alpha_5+\alpha_6+\alpha_8} +e_{\alpha_3+2\alpha_4+2\alpha_5+\alpha_6+\alpha_8}+e_{\alpha_1+\alpha_2+\alpha_3+\alpha_4+\alpha_5+\alpha_6+\alpha_7} \\ +e_{\alpha_4+2\alpha_5+2\alpha_6+\alpha_7+\alpha_8} + e_{\alpha_3+\alpha_4+2\alpha_5+\alpha_6+\alpha_7+\alpha_8}$;
	\item $e_{\alpha_1+\alpha_2+\alpha_3+\alpha_4+\alpha_5+\alpha_6+\alpha_7+\alpha_8}+e_{\alpha_2+\alpha_3+\alpha_4+2\alpha_5+\alpha_6+\alpha_7+\alpha_8}
	 + e_{\alpha_3+2\alpha_4+2\alpha_5+\alpha_6+\alpha_7+\alpha_8}\\+e_{\alpha_3+\alpha_4+2\alpha_5+2\alpha_6+\alpha_7+\alpha_8}$;
	\item $e_{\alpha_3+2\alpha_4+3\alpha_5+2\alpha_6+\alpha_7+2\alpha_8} + e_{\alpha_2+2\alpha_3+2\alpha_4 +2\alpha_5 +2\alpha_6+\alpha_7+\alpha_8}+e_{\alpha_2+\alpha_3+2\alpha_4+3\alpha_5+2\alpha_6+\alpha_7+\alpha_8}\\+e_{\alpha_1+\alpha_2+\alpha_3+2\alpha_4+2\alpha_5+2\alpha_6+\alpha_7+\alpha_8}$;
	\item $e_{\alpha_2+2\alpha_3+3\alpha_4+3\alpha_5+2\alpha_6+\alpha_7+\alpha_8}+e_{\alpha_2+2\alpha_3+2\alpha_4+3\alpha_5+2\alpha_6+\alpha_7+2\alpha_8}+e_{\alpha_1+2\alpha_2+2\alpha_3+2\alpha_4+2\alpha_5+2\alpha_6+\alpha_7+\alpha_8}\\+e_{\alpha_1+\alpha_2+2\alpha_3+2\alpha_4+3\alpha_5+2\alpha_6+\alpha_7+\alpha_8}$;
	\item $e_{\alpha_1+\alpha_2+2\alpha_3+3\alpha_4+3\alpha_5+2\alpha_6+\alpha_7+\alpha_8}+e_{\alpha_1+\alpha_2+2\alpha_3+2\alpha_4+3\alpha_5+2\alpha_6+\alpha_7+2\alpha_8}
	 + e_{\alpha_1+2\alpha_2+2\alpha_3+2\alpha_4+3\alpha_5+2\alpha_6+\alpha_7+\alpha_8}$;
	\item $e_{\alpha_2+2\alpha_3+3\alpha_4+4\alpha_5+3\alpha_6+2\alpha_7+2\alpha_8} + e_{\alpha_1+\alpha_2+2\alpha_3+3\alpha_4+4\alpha_5+3\alpha_6+\alpha_7+2\alpha_8}
	+ e_{\alpha_1+2\alpha_2+3\alpha_3+3\alpha_4+3\alpha_5+2\alpha_6+\alpha_7+2\alpha_8}
	\\+ e_{\alpha_1+2\alpha_2+2\alpha_3+3\alpha_4+4\alpha_5+2\alpha_6+\alpha_7+2\alpha_8}$;
	\item $e_{\alpha_1+2\alpha_2+3\alpha_3+4\alpha_4+4\alpha_5+2\alpha_6+\alpha_7+2\alpha_8}
	+ e_{\alpha_1+2\alpha_2+2\alpha_3+3\alpha_4+4\alpha_5+3\alpha_6+2\alpha_7+2\alpha_8}
	+ e_{\alpha_1+2\alpha_2+3\alpha_3+3\alpha_4+4\alpha_5+3\alpha_6+\alpha_7+2\alpha_8}$;
	\item $e_{\alpha_1+2\alpha_2+3\alpha_3+4\alpha_4+5\alpha_5+4\alpha_6+2\alpha_7+2\alpha_8}
	+ e_{\alpha_1+2\alpha_2+3\alpha_3+4\alpha_4+5\alpha_5+3\alpha_6+2\alpha_7+3\alpha_8}$;
	\item $e_{2\alpha_1+3\alpha_2+4\alpha_3+5\alpha_4+6\alpha_5+4\alpha_6+2\alpha_7+3\alpha_8}$.	
	
\end{enumerate}
In particular, $\dim(\fc_\fg(\chi))=16$ and so $d(\chi)=116=\left\vert \Phi^{+}\right\vert -4$. Proposition~\ref{nonspec} then says that all finite-dimensional $U_\chi(\fg)$-modules have dimension divisible by $2^{116}$.

\subsection{$E_8$ in characteristic 3}\label{E83}

Suppose $\Phi=E_8$ and $p=3$. Since $p$ is non-special in this case, we may apply Proposition~\ref{nonspec}. We must therefore give $c_\fg(\chi)$, and Sage computations show that $\fc_\fg(\chi)$ is the $\bK$-subspace of $\fg$ with the following basis:

\begin{enumerate}
	\item $e_{-\alpha_1-\alpha_2-2\alpha_3-2\alpha_4 - 2\alpha_5-\alpha_6-\alpha_8}+e_{-\alpha_1-\alpha_2-\alpha_3-2\alpha_4 -2\alpha_5-\alpha_6-\alpha_7-\alpha_8} + e_{-\alpha_2-2\alpha_3-2\alpha_4-2\alpha_5-\alpha_6-\alpha_7-\alpha_8}\\+2e_{-\alpha_1-\alpha_2-\alpha_3-\alpha_4-2\alpha_5-2\alpha_6-\alpha_7-\alpha_8}+2e_{-\alpha_3-2\alpha_4-3\alpha_5-2\alpha_6-\alpha_7-\alpha_8}+e_{-\alpha_2-\alpha_3-2\alpha_4-2\alpha_5-2\alpha_6-\alpha_7-\alpha_8}$;
	\item $e_{-\alpha_3-\alpha_4-\alpha_5-\alpha_6} + 2e_{-\alpha_3-\alpha_4-\alpha_5-\alpha_8} + e_{-\alpha_4-\alpha_5-\alpha_6-\alpha_8}
	 + e_{-\alpha_5-\alpha_6-\alpha_7-\alpha_8} + 2e_{-\alpha_4-\alpha_5-\alpha_6-\alpha_7}$;
	\item $e_{\alpha_1}+e_{\alpha_2}+e_{\alpha_3} +e_{\alpha_4} + e_{\alpha_5} +e_{\alpha_6}+e_{\alpha_7}+e_{\alpha_8}$;
	\item $e_{\alpha_1+\alpha_2+\alpha_3}+e_{\alpha_2+\alpha_3+\alpha_4} + e_{\alpha_3+\alpha_4+\alpha_5} + e_{\alpha_4+\alpha_5+\alpha_6} + e_{\alpha_4+\alpha_5+\alpha_8}+e_{\alpha_5+\alpha_6+\alpha_8} + e_{\alpha_5+\alpha_6+\alpha_7}$;
	\item $e_{\alpha_1+\alpha_2+\alpha_3+\alpha_4+\alpha_5+\alpha_6+\alpha_8} + 2e_{\alpha_2+\alpha_3+\alpha_4+2\alpha_5+\alpha_6+\alpha_8} + e_{\alpha_3+2\alpha_4+2\alpha_5+\alpha_6+\alpha_8}\\+2e_{\alpha_4+2\alpha_5+2\alpha_6+\alpha_7+\alpha_8}+e_{\alpha_2+\alpha_3+\alpha_4+\alpha_5+\alpha_6+\alpha_7+\alpha_8} + e_{\alpha_3+\alpha_4+2\alpha_5+\alpha_6+\alpha_7+\alpha_8}$;
	\item $e_{\alpha_1+\alpha_2+\alpha_3+\alpha_4+2\alpha_5+\alpha_6+\alpha_7+\alpha_8}+2e_{\alpha_1+\alpha_2+\alpha_3+2\alpha_4+2\alpha_5+\alpha_6+\alpha_8} +2e_{\alpha_2+\alpha_3+\alpha_4+2\alpha_5+2\alpha_6+\alpha_7+\alpha_8}\\+e_{\alpha_2+\alpha_3+2\alpha_4+2\alpha_5+\alpha_6+\alpha_7+\alpha_8}+e_{\alpha_2+2\alpha_3+2\alpha_4+2\alpha_5+\alpha_6+\alpha_8}$;
	\item $e_{\alpha_1+2\alpha_2+2\alpha_3+2\alpha_4+2\alpha_5+\alpha_6+\alpha_8}+e_{\alpha_1+\alpha_2+2\alpha_3+2\alpha_4+2\alpha_5+\alpha_6+\alpha_7+\alpha_8}
	 + e_{\alpha_3+2\alpha_4+3\alpha_5+2\alpha_6+\alpha_7+2\alpha_8}\\+e_{\alpha_2+2\alpha_3+2\alpha_4+2\alpha_5+2\alpha_6+\alpha_7+\alpha_8}+2e_{\alpha_2+\alpha_3+2\alpha_4+3\alpha_5+2\alpha_6+\alpha_7+\alpha_8} + e_{\alpha_1+\alpha_2+\alpha_3+2\alpha_4+2\alpha_5+2\alpha_6+\alpha_7+\alpha_8}$;
	\item $e_{\alpha_2+2\alpha_3+3\alpha_4+3\alpha_5+2\alpha_6+\alpha_7+\alpha_8} + 2e_{\alpha_2+2\alpha_3+2\alpha_4 +3\alpha_5 +2\alpha_6+\alpha_7+2\alpha_8}+e_{\alpha_1+2\alpha_2+2\alpha_3+2\alpha_4+2\alpha_5+2\alpha_6+\alpha_7+\alpha_8}\\+2e_{\alpha_1+\alpha_2+\alpha_3+2\alpha_4+3\alpha_5+2\alpha_6+\alpha_7+2\alpha_8}
	 + 2e_{\alpha_1+\alpha_2+2\alpha_3+2\alpha_4+3\alpha_5+2\alpha_6+\alpha_7+\alpha_8}$;
	\item $e_{\alpha_2+2\alpha_3+3\alpha_4+4\alpha_5+3\alpha_6+2\alpha_7+2\alpha_8}+2e_{\alpha_1+\alpha_2+2\alpha_3+3\alpha_4+4\alpha_5+3\alpha_6+\alpha_7+2\alpha_8}+2e_{\alpha_1+2\alpha_2+3\alpha_3+3\alpha_4+3\alpha_5+2\alpha_6+\alpha_7+2\alpha_8}\\+e_{\alpha_1+2\alpha_2+2\alpha_3+3\alpha_4+4\alpha_5+2\alpha_6+\alpha_7+2\alpha_8}$;
	\item $e_{\alpha_1+2\alpha_2+3\alpha_3+4\alpha_4+4\alpha_5+2\alpha_6+\alpha_7+2\alpha_8}+e_{\alpha_1+2\alpha_2+2\alpha_3+3\alpha_4+4\alpha_5+3\alpha_6+2\alpha_7+2\alpha_8}
	 \\+ 2e_{\alpha_1+2\alpha_2+3\alpha_3+3\alpha_4+4\alpha_5+3\alpha_6+\alpha_7+2\alpha_8}$
	\item $e_{\alpha_1+2\alpha_2+3\alpha_3+4\alpha_4+5\alpha_5+4\alpha_6+2\alpha_7+2\alpha_8} + 2e_{\alpha_1+2\alpha_2+3\alpha_3+4\alpha_4+5\alpha_5+3\alpha_6+2\alpha_7+3\alpha_8}$;
	\item $e_{2\alpha_1+3\alpha_2+4\alpha_3+5\alpha_4+6\alpha_5+4\alpha_6+2\alpha_7+3\alpha_8}$.
\end{enumerate}

Therefore $\dim(\fc_\fg(\chi))=12$ and so $d(\chi)=118=\left\vert \Phi^{+}\right\vert -2$. Proposition~\ref{nonspec} then says that each $U_\chi(\fg)$-module has dimension divisible by $3^{118}$.

\subsection{$E_8$ in characteristic 5}\label{E85}

Suppose $\Phi=E_8$ and $p=5$. Since $p$ is non-special in this case, we may apply Proposition~\ref{nonspec}. We must therefore give $c_\fg(\chi)$, and Sage computations show that $\fc_\fg(\chi)$ is the $\bK$-subspace of $\fg$ with the following basis:
\begin{enumerate}
	\item $e_{-\alpha_1-\alpha_2-\alpha_3-\alpha_4 - \alpha_5-\alpha_6}+4e_{-\alpha_1-\alpha_2-\alpha_3-\alpha_4 -\alpha_5-\alpha_8}  + e_{-\alpha_2-\alpha_3-\alpha_4-\alpha_5-\alpha_6-\alpha_8}+2e_{-\alpha_3-\alpha_4-2\alpha_5-\alpha_6-\alpha_8}\\+2e_{-\alpha_2-\alpha_3-\alpha_4-\alpha_5-\alpha_6-\alpha_7}+2e_{-\alpha_4-2\alpha_5-\alpha_6-\alpha_7-\alpha_8}  + 3e_{-\alpha_3-\alpha_4-\alpha_5-\alpha_6-\alpha_7-\alpha_8}$;
	\item $e_{\alpha_1}+e_{\alpha_2}+e_{\alpha_3} +e_{\alpha_4} + e_{\alpha_5} +e_{\alpha_6}+e_{\alpha_7}+e_{\alpha_8}$;
	\item $e_{\alpha_1+\alpha_2+\alpha_3+\alpha_4+\alpha_5}+e_{\alpha_2+\alpha_3+\alpha_4+\alpha_5+\alpha_6} + e_{\alpha_2+\alpha_3+\alpha_4+\alpha_5+\alpha_8} + 2e_{\alpha_4+2\alpha_5+\alpha_6+\alpha_8} + 3e_{\alpha_3+\alpha_4+\alpha_5+\alpha_6+\alpha_8}\\ +e_{\alpha_3+\alpha_4+\alpha_5+\alpha_6+\alpha_7}+2e_{\alpha_4+\alpha_5+\alpha_6+\alpha_7+\alpha_8}$;
	\item $e_{\alpha_1+\alpha_2+\alpha_3+\alpha_4+\alpha_5+\alpha_6+\alpha_8} + 4e_{\alpha_2+\alpha_3+\alpha_4+2\alpha_5+\alpha_6+\alpha_8} + e_{\alpha_3+2\alpha_4+2\alpha_5+\alpha_6+\alpha_8}+3e_{\alpha_1+\alpha_2+\alpha_3+\alpha_4+\alpha_5+\alpha_6+\alpha_7}\\+4e_{\alpha_4+2\alpha_5+2\alpha_6+\alpha_7+\alpha_8}+3e_{\alpha_2+\alpha_3+\alpha_4+\alpha_5+\alpha_6+\alpha_7+\alpha_8} + e_{\alpha_3+\alpha_4+2\alpha_5+\alpha_6+\alpha_7+\alpha_8}$;
	\item $e_{\alpha_1+2\alpha_2+2\alpha_3+2\alpha_4+2\alpha_5+\alpha_6+\alpha_8}+e_{\alpha_1+\alpha_2+2\alpha_3+2\alpha_4+2\alpha_5+\alpha_6+\alpha_7+\alpha_8} + 2e_{\alpha_3+2\alpha_4+3\alpha_5+2\alpha_6+\alpha_7+2\alpha_8}\\+2e_{\alpha_3+2\alpha_4+3\alpha_5+2\alpha_6+\alpha_7+2\alpha_8}+3e_{\alpha_2+\alpha_3+2\alpha_4+3\alpha_5+2\alpha_6+\alpha_7+\alpha_8}+2e_{\alpha_1+\alpha_2+\alpha_3+2\alpha_4+2\alpha_5+2\alpha_6+\alpha_7+\alpha_8}$;
	\item $e_{\alpha_2+2\alpha_3+3\alpha_4+3\alpha_5+2\alpha_6+\alpha_7+\alpha_8}+4e_{\alpha_2+2\alpha_3+2\alpha_4+3\alpha_5+2\alpha_6+\alpha_7+2\alpha_8} + e_{\alpha_1+2\alpha_2+2\alpha_3+2\alpha_4+2\alpha_5+2\alpha_6+\alpha_7+\alpha_8}\\+2e_{\alpha_1+\alpha_2+\alpha_3+2\alpha_4+3\alpha_5+2\alpha_6+\alpha_7+2\alpha_8}+4e_{\alpha_1+\alpha_2+2\alpha_3+2\alpha_4+3\alpha_5+2\alpha_6+\alpha_7+\alpha_8}$;
	\item $e_{\alpha_2+2\alpha_3+3\alpha_4+4\alpha_5+3\alpha_6+2\alpha_7+2\alpha_8} + 4e_{\alpha_1+\alpha_2+2\alpha_3+3\alpha_4 +4\alpha_5 +3\alpha_6+\alpha_7+2\alpha_8}+4e_{\alpha_1+2\alpha_2+3\alpha_3+3\alpha_4+3\alpha_5+2\alpha_6+\alpha_7+2\alpha_8}\\+e_{\alpha_1+2\alpha_2+2\alpha_3+3\alpha_4+4\alpha_5+2\alpha_6+\alpha_7+2\alpha_8}$;
	\item $e_{\alpha_1+2\alpha_2+3\alpha_3+4\alpha_4+4\alpha_5+2\alpha_6+\alpha_7+2\alpha_8}+e_{\alpha_1+2\alpha_2+2\alpha_3+3\alpha_4+4\alpha_5+3\alpha_6+2\alpha_7+2\alpha_8}\\+4e_{\alpha_1+2\alpha_2+3\alpha_3+3\alpha_4+4\alpha_5+3\alpha_6+\alpha_7+2\alpha_8}$;
	\item $e_{\alpha_1+2\alpha_2+3\alpha_3+4\alpha_4+5\alpha_5+4\alpha_6+2\alpha_7+2\alpha_8}+4e_{\alpha_1+2\alpha_2+3\alpha_3+4\alpha_4+5\alpha_5+3\alpha_6+2\alpha_7+3\alpha_8}$;
	\item $e_{2\alpha_1+3\alpha_2+4\alpha_3+5\alpha_4+6\alpha_5+4\alpha_6+2\alpha_7+3\alpha_8}$.
\end{enumerate}

In particular we see that $\dim c_{\fg}(\chi)=10$, and so $d(\chi)=119=\left\vert\Phi^{+}\right\vert-1$. Hence, every finite-dimensional $U_\chi(\fg)$-module has dimension divisible by $5^{119}$.

\end{document}